\documentclass[a4paper,reqno,11pt]{amsart}
\usepackage[a4paper,margin=1in]{geometry}
\usepackage{amsmath,amssymb,amsthm}
\theoremstyle{plain}
\newtheorem{theorem}{Theorem}[section]
\newtheorem{lemma}[theorem]{Lemma}
\newtheorem{proposition}[theorem]{Proposition}

\newtheorem{corollary}[theorem]{Corollary}
\theoremstyle{definition}
\newtheorem{remark}[theorem]{Remark}
\numberwithin{equation}{section}

\title[Subharmonic bifurcations]{On first subharmonic bifurcations in a branch of Stokes waves.}

\author{Vladimir Kozlov$^1$}
\address{$^1$Department of Mathematics, Link\"oping University, SE-581 83 Link\"oping, Sweden}

\begin{document}
	
\begin{abstract} Steady surface waves in a two-dimensional channel are considered. We study bifurcations, which occur on a branch of Stokes water waves starting from a uniform stream solution. Two types of bifurcations are considered: bifurcations in the class of Stokes waves (Stokes bifurcation) and bifurcations in a class of periodic waves with the period $M$ times the period of the Stokes wave ($M$-subharmonic bifurcation). If we consider the first Stokes bifurcation point
 then there are no $M$-subharmonic bifurcations before this point and there exists  $M$-subharmonic bifurcation points after the first Stokes bifurcation for sufficiently large $M$, which approach the Stokes bifurcation point when $M\to\infty$. Moreover the set of $M$-subharmonic bifurcating solutions is a closed connected continuum. We give also a more detailed description of this connected set in terms of the set of its limit points, which must contain extreme waves, or overhanging waves, or solitary waves or waves with stagnation on the bottom, or Stokes bifurcation points different from the initial one.

\end{abstract}

\maketitle

\section{Introduction}

 \subsection{Background}

We consider steady surface waves in a two-dimensional channel bounded below by a flat,
rigid bottom and above by a free surface.
The main subject of this work is subharmonic bifurcations on the branches of Stokes waves.

Stokes and solitary waves (regular waves) were the main subject of study up to 1980.
In 1980 (see Chen \cite{Che} and Saffman \cite{Sa}) it was discovered numerically and in 2000 (see \cite{BDT1,BDT2}) this was supported theoretically for the ir-rotational case for a flow of infinite depth that there exist new types of periodic waves with several crests on the period (the Stokes wave has only one crest). These waves occur as a result of bifurcation on a branch of Stokes waves when they approach the wave of greatest amplitude. 


 Starting point of our study is a branch of Stokes waves starting from uniform stream solution and approaching an extreme wave. This branch is parameterized by a parameter $t$. One can study bifurcations of branches of Stokes waves of period $\Lambda(t)$ in
 one of the following settings:

(i) in the class of $\Lambda(t)$-periodic solutions (Stokes bifurcation);

(ii) in the class of $M\Lambda(t)$-periodic solutions (M-subharmonic bifurcation);

(iii) in the class of bounded solutions.

\noindent
In this paper we will deal mainly with (i) and (ii).

Let us denote by $\{\tilde{t}_j\}$, $j=1,\ldots,\infty$, the Stokes bifurcation points. The following theorem is proved in \cite{Koz1}

\begin{theorem}\label{Tmain22} {\rm (i)} There exists a sequence $(\widehat{t}_j, M_j)$, where $\widehat{t}_j\neq \tilde{t}_k$ for all $j$ and $k$; $M_j$ are integers, and
$$
\widehat{t}_j\to\infty,\;\;M_j\to\infty\;\;\mbox{as $j\to\infty$}.
$$
Moreover, $\widehat{t}_j$ is a $M_j$-subharmonic bifurcation point.

{\rm (ii)} There exists a sequence $(\widehat{t}_j, M_j)$, where $\widehat{t}_j\neq \tilde{t}_k$ for all $j$ and $k$, the sequence $\{\widehat{t}_j\}$ is bounded and
$$
M_j\to\infty\;\;\mbox{as $j\to\infty$}.
$$
Furthermore, $\widehat{t}_j$ is a a $M_j$-subharmonic bifurcation point.
The numbers $\widehat{t}_j$ are pairwise different in both cases.
\end{theorem}

The proof of this theorem is based on the fact that when the Stokes waves approach  the extreme wave many new bifurcation points appear. The proof of this theorem involves  a bifurcation theorem for potential operators with non-zero crossing number, which gives unfortunately no information on the structure of the set of bifurcating solutions (see \cite{To1} for a discussion of this structure).

Here we study a mechanism of appearance of the first subharmonic bifurcations. The main result of this paper is the following

\begin{theorem} Let $t_0>0$ be the first Stokes bifurcation point, i.e. the condition {\rm (\ref{KKK1})} be valid. Then there exists  $(t_M,M)$, where $M$ is a large integer and
$$
t_M\to t_0\;\;\;\mbox{as}\;\; M\to\infty,
$$
 such that $t_M$ is $M$- subharmonic bifurcation point. There are no subharmonic bifurcations for $t<t_0$.
\end{theorem}

We prove this theorem by using a bifurcation theorem with odd crossing number, which leads to existence of a closed connected continuum of $M\Lambda$-periodic solutions containing the bifurcation point $t_M$  (see Theorems \ref{TF9a}, \ref{T4.2C} and \ref{ThA22a} for more details).

In the study of bifurcations the Frechet derivative plays important role. In our case it is a self-adjoint operator bounded from below.
 Certainly the Frechet derivatives are defined on different spaces in the cases (i)-(iii). In the study of the Stokes bifurcations  this is a usual Frechet derivative defined on $\Lambda$-periodic, even solutions. If we denote by $\mu_0(t)$ and $\mu_1(t)$ the first and the second eigenvalue of the Frechet derivative, then $\mu_0(t)<0$ for all $t$ and $\mu_1(0)=0$. We consider the case when
 \begin{eqnarray}\label{KKK1}
 &&\mu_1(t)>0\;\; \mbox{for small positive $t$ and assume that there exists $t_0>0$}\nonumber\\
&&\mbox{ such that $\mu_1(t_0)=0$ and $\mu_1(t)<0$ for small positive $t-t_0$}.
 \end{eqnarray}
  We will call $t_0$ the first Stokes bifurcation.

 In the case of $M$ -subharmonic bifurcations  we have the same expression for the Frechet derivative but now it is defined on $M\Lambda$-periodic, even functions and finally  in the case of bifurcations in a class of bounded solutions the Frechet derivative is defined on the even functions defined on the whole domain without periodicity condition.
By introducing quasi-momentum $\tau$ we can give another equivalent definition of the Frechet derivative in the cases (ii) and (iii) which is defined on the space of $\Lambda$-periodic functions but depending on the real parameter $\tau$. We denote by $\widehat{\mu}_j(t,\tau)$, $j=0,\ldots$,  the eigenvalues of this problem, which are numerated according to the increasing order:
$$
\widehat{\mu}_0(t,\tau)\leq\widehat{\mu}_1(t,\tau)\leq\cdots
$$
The essential part of the paper is devoted to a study of properties of these eigenvalues. The eigenvalue $\widehat{\mu}_0(t,\tau)$ is always less than $0$ for all $t$ and $\tau\in\Bbb R$. If the first Stokes bifurcation occurs at a certain $t=t_0$, then it is proved that the kernel of the Frechet derivative at this point is one-dimensional. Another important property concerning the Frechet derivative is the inequalities
$$
\widehat{\mu}_1(t,\tau)<0\;\;\mbox{for $t\in (t_0,t_0+\epsilon)$ and $\tau\in\Bbb R$},
$$
where $\epsilon$ is a small positive number, and
$$
\widehat{\mu}_2(t_0,\tau)>0\;\;\mbox{for $\tau\in (0,\tau_*)$},\;\;\tau_*=2\pi/\Lambda.
$$
These inequalities together with asymptotic properties of $\widehat{\mu}_2$ in a neighborhood of the point $(t,\tau)=(t_0,0)$ allows to perform analysis of subharmonic bifurcations near the point $t=t_0$. $M$-subharmonic bifurcation points $t_M$ are bigger then $t_0$ and $t_M\to t_0$ as $M\to\infty$.  Moreover, we show that the crossing number at $t_M$ is odd and hence there is a closed connected continuum of $M $-subharmonic bifurcations started from the bifurcation point.

 As is known one of the indication of bifurcation is the change of the Morse index, or changing of the number of negative eigenvalues of the Frechet derivatives, which depend on the parameter $t$. In the papers \cite{KL2} and \cite{Koz1}  it was proved that for the branch of Stokes waves approaching the extreme Stokes wave unlimited number of negative eigenvalues appears, which implies unlimited number of bifurcations, see \cite{Koz1}. The problem here is that all these bifurcations can be Stokes bifurcations and the above fact on negative eigenvalues does not directly imply existence of subharmonic bifurcations. More advance analysis is required.

Our aim is to study the first Stokes bifurcation and show that there are always accompanying $M$-subharmonic bifurcations with arbitrary large $M$. They generates by $\tau$-roots of the equation $\widehat{\mu}_2(t,\tau)=0$. The main body of the paper is devoted to the study of this equation. This analysis is based on spectral analysis of different spectral problems and asymptotic analysis of eigenvalues for small $\tau$.

\subsection{Formulation of the problem}

We consider steady surface waves in a two-dimensional channel bounded below by a flat,
rigid bottom and above by a free surface that does not touch the bottom. The surface tension is neglected and the water motion can be rotational.
In appropriate Cartesian coordinates $(X, Y )$, the bottom coincides with the
$x$-axis and gravity acts in the negative $y$ -direction. We choose the frame of reference so that the velocity field is time-independent as well as the free-surface profile
which is supposed to be the graph of $Y = \xi(X)$, $X \in \Bbb R$, where $\xi$ is a positive and
continuous unknown function. Thus
$$
{\mathcal D}_\xi = \{X\in \Bbb R, 0 <Y < \xi(X)\},\;\;{\mathcal S}_\xi=\{X\in\Bbb R,\;Y=\xi(X)\}
$$
is the water domain and the free surface respectively. We will use the stream function $\Psi$, which is connected with the velocity vector $({\bf u},{\bf v})$ by ${\bf u}=-\Psi_Y$ and ${\bf v}=\Psi_X$.

We assume that $\xi$ is a positive,  periodic function having period $\Lambda>0$ and that $\xi$   is even and strongly decreasing on the interval $(0,\Lambda/2)$.  Since the surface tension is neglected, $\Psi$ and
$\xi$  satisfy, after a certain scaling, the following free-boundary problem (see for example \cite{KN14}):
\begin{eqnarray}\label{K2a}
&&\Delta \Psi+\omega(\Psi)=0\;\;\mbox{in ${\mathcal D}_\xi$},\nonumber\\
&&\frac{1}{2}|\nabla\Psi|^2+\xi=R\;\;\mbox{on ${\mathcal S}_\xi$},\nonumber\\
&&\Psi=1\;\;\mbox{on ${\mathcal S}_\xi$},\nonumber\\
&&\Psi=0\;\;\mbox{for $Y=0$},
\end{eqnarray}
where $\omega\in C^{1,\alpha}$, $\alpha\in (0,1)$, is a vorticity function and $R$ is the Bernoulli constant. We assume that $\Psi$ is even, $\Lambda$-periodic in $x$ and
\begin{equation}\label{J27ba}
\Psi_Y>0\;\;\mbox{on $\overline{{\mathcal D}_\xi}$},
\end{equation}
which means that the flow is unidirectional.

The Frechet derivative for the problem is evaluated for example in \cite{KL2}, \cite{Koz1},  and the corresponding eigenvalue problem for the  Frechet derivative has the form
\begin{eqnarray}\label{J17ax}
&&\Delta w+\omega'(\Psi)w+\mu w=0 \;\;\mbox{in ${\mathcal D}_\xi$},\nonumber\\
&&\partial_\nu w-\rho w=0\;\;\mbox{on ${\mathcal S}_\xi$},\nonumber\\
&&w=0\;\;\mbox{for $Y=0$},
\end{eqnarray}
where $\nu$ is the unite outward normal to ${\mathcal S}_\xi$ and
\begin{equation}\label{Sept17aa}
\rho=
\rho(X)=\frac{(1+\Psi_X\Psi_{XY}+\Psi_Y\Psi_{YY})}{\Psi_Y(\Psi_X^2+\Psi_Y^2)^{1/2}}\Big|_{Y=\xi(X)}.
\end{equation}
The function $w$ in (\ref{J17ax}) is supposed also to be even and $\Lambda$-periodic. 

Let us introduce several function spaces. Let $\alpha\in (0,1)$ and $k=0,1,\ldots$. The space $C^{k,\alpha}({\mathcal D})$ consists of bounded  functions in ${\mathcal D}$ such that the norms $C^{k,\alpha}(\overline{{\mathcal D}_{a,a+1}})$ are uniformly bounded with respect to $a\in\Bbb R$. Here
$$
{\mathcal D}_{a,a+1}=\{(X,Y)\in {\mathcal D},\;:\,a<x<a+1\}.
$$
The space   $C^{k,\alpha}_{0,\Lambda }({\mathcal D})$ \big($C^{k,\alpha}_{0,\Lambda, e}({\mathcal D})\big)$ consists of $\Lambda$-periodic ($\Lambda$-periodic and even) functions, which belong to $C^{k,\alpha}({\mathcal D})$ and vanish at $y=0$.

Similarly we define the space  $C^{k,\alpha}_{\Lambda}(\Bbb R)$  ($C^{k,\alpha}_{\Lambda, e}(\Bbb R)$) consisting of functions in $C^{k,\alpha}(\Bbb R)$, which are $\Lambda$-periodic ($\Lambda$-periodic and even).

We will consider a branch of Stokes water waves depending on a parameter $t\in\Bbb R$, i.e.
$$
\xi=\xi(X,t),\;\;\Psi=\Psi(X,Y;t),\;\;\Lambda=\Lambda(t).
$$
For each $t$ the function $\xi\in C^{2,\alpha}(\Bbb R)$ and $\Psi\in C^{3,\alpha}({\mathcal D})$. This branch starts from a uniform stream solution for $t=0$ and approach an extreme wave when $t\to\infty$. The dependence on $t$ is analytic in the following sense: the functions
$$
\xi(X\Lambda(0)/\Lambda(t),t)\,:\,\Bbb R\rightarrow C^{2,\alpha}_{\Lambda(0),e}(\Bbb R)\;\;\mbox{and}\;\;\Lambda\,:\,\Bbb R\rightarrow (0,\infty)
$$
are analytic with respect to $t$ and the function $\Psi$ can be found from the problem
\begin{eqnarray*}
&&\Delta \Psi+\omega(\Psi)=0\;\;\mbox{in ${\mathcal D}_\xi$},\\
&&\Psi=1\;\;\mbox{on ${\mathcal S}_\xi$},\nonumber\\
&&\Psi=0\;\;\mbox{for $Y=0$}.
\end{eqnarray*}
Another equivalent description of the analytical property of functions $\xi$ and $\Psi$ is presented in Sect. \ref{SF22}.

The main assumption is given in terms of the second eigenvalue of the spectral problem (\ref{J17ax}). One can show that the first eigenvalue $\mu_0(t)$ is always negative and the second one $\mu_1(t)$ is zero for $t=0$. Our main assumption is that
\begin{equation}\label{F14a}
\mu_1(t)>0 \;\;\mbox{for small positive $t$}.
\end{equation}
This property will be discussed in detail in forthcoming paper.

\section{Spectral problem (\ref{J17ax}), $\Lambda$-periodic, even functions}\label{SJ24}

Let
$$
D=D_\xi=\{ (X,Y)\,:\,0< X< \Lambda/2,\;0<X<\xi(X)\}
$$
and
$$
S=S_\xi=\{ (X,Y)\,:\,0< X< \Lambda/2,\;Y=\xi(X)\}.
$$
We introduce the form
\begin{equation}\label{Ja22b}
a(u,v)\!\!=\!\!a_D(u,v)\!\!=\!\!\int_{D}\Big(\nabla u\cdot\nabla\overline{v}-\omega'(\Psi)u\overline{v}\Big)dXdY-\int_{0}^{\Lambda/2}\rho(X)u\overline{v}dS,\;\;dS=\sqrt{1+\xi'^2}dX,
\end{equation}
which is defined on  functions from $H^1(D)$ vanishing for $Y=0$. This space will be denoted by $H^1_0(D)$.

The assumption (\ref{J27ba}) implies
\begin{equation}\label{J27c}
a_D(u,u)>0\;\;\mbox{for nonzero $u\in H^1_0(D)$ satisfying $u(X,\xi(X))=0$ for $X\in (0,\Lambda/2)$}.
\end{equation}

Since $\partial_Xw(0,Y)=\partial_Xw(\Lambda/2,Y)=0$ in (\ref{J17ax}), the spectral problem (\ref{J17ax}) admits the following variational formulation. A real number $\mu$ and  a function $\phi \in H^1_0(D)$ satisfy (\ref{J17ax}), after natural extension on ${\mathcal D}_\xi$, if and only if
\begin{equation}\label{Sept8a}
a(\phi,v)=\mu (\phi,v)_{L^2(D)},\;\;\mbox{for all $v\in H^1_0(D)$}.
\end{equation}
Here
$$
(u,v)_{L^2(D)}=\int_Du\overline{v}dXdY.
$$
We put also $||u||_{L^2(D)}^2=(u,u)_{L^2(D)}$.
We numerate the eigenvalues accounting their multiplicity as
$$
\mu_0\leq \mu_1\leq\cdots ,\;\;\mu_j\to\infty\;\;\mbox{as}\;\;j\to\infty,
$$
and denote by $\Phi_j$, $j=0,1,\ldots$, corresponding eigenfunctions. They can be chosen to be orthogonal to each other, real-valued and with the norm
$$
||\Phi_j||_{L^2(D)}=1.
$$
Thus $\{\Phi_j\}$ forms an orthogonal basis in $L^2(D)$.

In what follows we shall use the function
\begin{equation}\label{Sept13b}
u_*(X,Y):=\Psi_X(X,Y),
\end{equation}
which is odd, $\Lambda$-periodic and belongs to $C^{2,\alpha}(\overline{D})$. It satisfies the boundary value problem
\begin{eqnarray}\label{Okt21a}
&&\Delta u_*+\omega'(\Psi)u_*=0 \;\;\mbox{in $D_\xi$},\nonumber\\
&&\partial_\nu u_*-\rho u_*=0\;\;\mbox{on $S_\xi$},\nonumber\\
&&u_*=0\;\;\mbox{for $X=0$},\nonumber\\
&&u_*(0,Y)=0\;\;\mbox{for $Y\in (0,\xi(0))$},\;\;u_*(\Lambda/2,Y)=0\;\;\mbox{for $Y\in (0,\xi(\Lambda/2))$},
\end{eqnarray}
and it has the following properties along branches of Stokes waves (see \cite{Koz2})
\begin{eqnarray}\label{Au16a}
&&{\rm (i)}\;\;\;\;u_*>0\;\; \mbox{for $(X,Y)\in D\cup S$};\nonumber\\
&&{\rm (ii)}\;\;\;(u_*)_X(0,Y)>0\;\; \mbox{for $0<Y\leq \xi(0)$,\;\; $(u_*)_X(\Lambda/2,Y)<0$ for $0<Y\leq \xi(\Lambda/2)$};\nonumber\\
&&\;\;\;\;\;\;\;\;\mbox{ and \;\;$(u_*)_Y(X,0)>0$\; for $0<X<\Lambda/2$};\nonumber\\
&&{\rm (iii)}\;\;(u_*)_{XY}(0,0)>0,\;\; (u_*)_{XY}(\Lambda/2,0)<0.
\end{eqnarray}
The above properties of the function $\Psi_x$-the vertical component of the velocity vector, are found interesting application in study of branches of Stokes waves, see \cite{CSst}, \cite{CSrVar}, \cite{Var23}, \cite{Wa} and \cite{Koz2}. It appears that they are also important in study of subharmonic bifurcations. They are used, in particular, in the proof of the following assertion.

\begin{proposition}\label{Pr1}  Let $\xi$ be not identically constant then the eigenvalue $\mu_0$ is simple \footnote{This implies in particular that $\mu_0<\mu_1$.}. Moreover
$\mu_0<0$ and the corresponding eigenfunction ($\Phi_0$) does not change sign inside $D$. If we assume that $\Phi_0$ is positive inside $D$ then $\Phi_0$ is also positive for $Y=\xi(X)$ and for $X=0$, $Y\in (0,\xi(0))$ and $X=\Lambda/2$, $Y\in (0,\xi(0))$, $Y\in (0,\xi(\Lambda/2))$.
\end{proposition}
\begin{proof} Let $v_*$ be the restriction of the function (\ref{Sept13b}) to the domain $D$. Clearly it is positive inside $D$,  belongs to $H^1_0(D)$ and $a(v_*,v_*)=0$. This implies that either $\mu_0<0$ or $\mu_0=0$ and the function $v_*$ is an eigenfunction  corresponding to $\mu_0$. The latest alternative is impossible. Indeed, if $v_*$ is an eigenfunction then $(v_*)_X=0$ for $X=0$ and for $X=\Lambda/2$. So $v_*$ has homogeneous Cauchy data and satisfies an elliptic equation. Hence $v_*$ is identically zero inside $D$. Thus $\mu_0<0$.

 Since
\begin{equation}\label{Okt10a}
\mu_0=\min_{w\in H^1_0(D), ||w||_0=1} a(w,w)
\end{equation}
the corresponding eigenfunction cannot change sign inside $D$. Indeed, if $\phi_0$ changes sign then we introduce two functions $v_\pm(X,Y)=\max (0,\pm v(X,Y))$. One can verify that $a(v_\pm,v_\pm)=\mu_0||v_\pm||^2_{L^2(D)}$. Therefore both functions $v_\pm$ deliver minimum in (\ref{Okt10a}) and hence they are also eigenfunctions corresponding to $\mu_0$ but this is impossible because they are not smooth. So $\Phi_0$ is positive or negative inside $D$.

Assume that $\Phi_0>0$ in $D$. Then it cannot vanish on the part of the boundary where $Y>0$. Indeed if $\Phi$ is zero there then the normal derivative is also zero and this leads to a chenge of sign inside $D$.
\end{proof}

The following Green's formula
 for the form $a$ will be useful in the proof of the nest proposition:
\begin{eqnarray}\label{Okt31a}
&&-\int_D(\Delta u+\omega'(\Psi)u)\overline{v}dXdY+\int_0^{\Lambda/2}(\partial_\nu u-\rho(X)u)\overline{v}dS\nonumber\\
&&=a(u,v)+\int_0^{\xi(0)}u_X\overline{v}|_{X=0}dY-\int_0^{\xi(\Lambda/2)}u_X\overline{v}|_{X=\Lambda/2}dY.
\end{eqnarray}
Here the function $u$ belongs to the space $H^2_0(D)$ consisting of functions from $H^2(D)$ vanishing for $Y=0$ and $v\in H^1_0(D)$.

\begin{proposition}\label{PF10a} Let $\xi$ be not identically constant. If $\mu_1=0$ then  $\mu_1=0$ is simple. Moreover the corresponding eigenfunction has exactly two nodal sets and the nodal line separating these nodal sets has one of its end points on the curve $S$.
\end{proposition}
\begin{proof} Denote by ${\mathcal X}$ the  space of real eigenfunctions corresponding to the eigenvalue $\mu_1=0$.  Since all eigenfunctions from ${\mathcal X}$ orthogonal to $\Phi_0$, each eigenfunction must change sign inside $D$. Let us show that $\dim {\mathcal X}=1$. The proof consists of several steps.

(a) First we show that every nonzero eigenfunction $w\in {\mathcal X}$ has exactly two nodal sets. Assume that there is an eigenfunction $w\in {\mathcal X}$ which has  more then two nodal sets, say $Y_j$, $j=1,\ldots, N$, $N>2$. Introduce the functions $u_j(X,Y)=w(X,Y)$ for $(X,Y)\in Y_j$ and zero otherwise. Then $u_j\in H^1_0(D)$ and
\begin{equation}\label{Sept13a}
a(u_j,u_k)=0\;\;\mbox{for all $k,j=1,\ldots,N$}.
\end{equation}
We choose the constant $\alpha$ such that
$$
(u_1-\alpha u_2,\Phi_0)_{L^2(D)}=0
$$
and observe that $a(u_1-\alpha u_2,u_1-\alpha u_2)=0$. Since
$$
\mu_1=\min a(w,w)
$$
where $\min$ is taken over $w\in H^1_0(D)$ satisfying $(w,\Phi_0)_{L^2(D)}=0$ and $||w||_{L^2(D)}=1$, the minimum is attained on the function $(u_1-\alpha u_2)/||u_1-\alpha u_2||_{L^2(D)}$. So we conclude that $u_1-\alpha u_2$ is an eigenfunction corresponding to the eigenvalue $\mu_1=0$, which is impossible since this function is zero on $Y_j$ for $j>2$.

  (b) Second, let us prove that the nodal line is not closed and one of its end points lies oh $S$.
  Consider an eigenfunction $w$ with two nodal sets and let $\gamma$ be the nodal line separating these two nodal sets.
  If this nodal line is closed then introduce the function $w_1$ which coincides with $w$ inside the closed nodal line and vanishes outside.
  Then $a(w_1,w_1)=0$ and $w_1=0$ on $S$ which contradicts to  the inequality (\ref{J27c}).
  If both end points of $\gamma$ lie outside $S$ then introduce $w_2$ which coincides with $w$ on the nodal set separated from $S$ and vanishes otherwise. Then $a(w_2,w_2)=0$ and $w_2=0$ on $S$ which again contradicts to the inequality (\ref{J27c}).
 Thus one of end points lies on $S$.

(c) Now we are in position to prove that $\dim {\mathcal X}=1$. If $\dim {\mathcal X}>1$ then there is an eigenfunction, say $w_*$ which is zero at the point $z_1=(0,\eta(0))$, which must be one of the end points of the nodal line separating two nodal sets. By b) another end-point $z_2$ of the nodal line lies on $S$. Denote by $Y_1$ the nodal set attached to the part of $S$ between $z_1$ and $z_2$. Let also $Y_2$ be the remaining nodal domain. We can assume that $w_*<0$ in $Y_1$ and $w_*>0$ in $Y_2$.

Let $v_*$ be the function introduced in Proposition \ref{Pr1}.  Using that both functions $w_*$ and $u_*$ satisfy the problem (\ref{J17ax}) but the first  one satisfies  $\partial_Xw_*=0$ for $X=0,\Lambda/2$ and the second is subject to $v_*=0$ for the same values of $x$, we have, by using (\ref{Okt31a}),
\begin{equation}\label{Sept13aa}
\int_0^{\xi(0)}\partial_Xv_*(0,Y)w_*(0,Y)dY=\int_0^{\xi(\Lambda/2)}\partial_Xv_*(\Lambda/2,Y)w_*(\Lambda/2,Y)dY.
\end{equation}

Consider the function
$$
U=w_*+\beta v_*,\;\;\mbox{where $\beta>0$}.
$$
Since $v_*$ is a positive function inside $D$, for small $\beta$ the function $U$ has also two nodal sets $\widetilde{Y}_1$ and $\widetilde{Y}_2$, separated by a nodal line $\widetilde{\gamma}$ with end points $z_1=(0,\eta(0))$ and $\widetilde{z}_2\in S$ which is close to $z_2$. We assume that the nodal set $\widetilde{Y}_1$ is attached to the part of $S$ between the points $z_1$ and $\widetilde{z}_2$. Introduce the functions
\begin{eqnarray*}
&&U_1(X,Y)=U(X,Y)\;\;\mbox{if $(X,Y)\in \widetilde{Y}_1$ and zero otherwise},\\
&&U_2(X,Y)=U(X,Y)\;\;\mbox{if $(X,Y)\in \widetilde{Y}_2$ and zero otherwise}.
\end{eqnarray*}
 We choose the constant $\theta$ such that
 \begin{equation}\label{Sept13ba}
 (U_1+\theta U_2,\Phi_0)_{L^2(D)}=0.
 \end{equation}
 It is clear that $\theta\neq 0$ because of
$$
a(U_1,\Phi_0)=\mu_0(U_1,\Phi_0)_{L^2(D)}\neq 0.
$$
Here we have used that $\Phi_0>0$ inside $D$ by Proposition \ref{Pr1}. Furthermore due to (\ref{Sept13aa}) and the fact that $w_*$ is an eigenfunction corresponding to $\mu_1=0$, we have
$$
a(U_1+\theta U_2,U_1+\theta U_2)=0.
$$
This together with (\ref{Sept13ba}) implies that $U_1+\theta U_2$ is an eigenfunction corresponding to the eigenvalues $\mu_1=0$ of the spectral problem (\ref{Sept8a}). But this contradicts to the smoothness properties of the function $U_1+\theta U_2$. (Compare with the argument used in a) and b)).
\end{proof}

\subsection{Auxiliary spectral problems}

In this section we introduce and study eigenvalues of several eigenvalue problems, which will be used in subsequent sections to estimate eigenvalues of "generalized" eigenvalue problem introduced in the next section. Zeros of  generalized eigenvalues will give us the subharmonic bifurcations in what follows.

Let us introduce the spaces
$$
H^1_{00*}(D)=\{u\in H^1_0(D)\,:\;u=0\;\mbox{for $X=0$}\},
$$
$$
H^1_{0*0}(D)=\{u\in H^1_0(D)\,:\;u=0\;\mbox{for $X=\Lambda/2$}\}
$$
and
$$
H^1_{000}(D)=\{u\in H^1_0(D)\,:\;u=0\;\mbox{for $X=0$ and $X=\Lambda/2$} \}.
$$
The following spectral problems will play important role in what follows. Find $\nu^{0*}$ and $u\in H^1_{00*}(D)$ satisfying
\begin{equation}\label{Ja24a}
a(u,v)=\nu^{0*}(u,v)_{L^2(D)}\;\;\mbox{for all $v\in H^1_{00*}(D)$}.
\end{equation}
Find $\nu^{*0}$ and $u\in H^1_{0*0}(D)$ satisfying
\begin{equation}\label{Ja24aa}
a(u,v)=\nu^{*0}(u,v)_{L^2(D)}\;\;\mbox{for all $v\in H^1_{0*0}(D)$}.
\end{equation}
Find $\nu^{00}$ and $u\in H^1_{000}(D)$ satisfying
\begin{equation}\label{Ja24ab}
a(u,v)=\nu^{00}(u,v)_{L^2(D)}\;\;\mbox{for all $v\in H^1_{000}(D)$}.
\end{equation}
We denote the eigenvalues of these problems by
$$
\nu^{0*}_0\leq \nu^{0*}_1\leq\cdots,
$$
$$
\nu^{*0}_0\leq \nu^{*0}_1\leq\cdots
$$
and
$$
\nu^{00}_0\leq \nu^{00}_1\leq\cdots,
$$
where the multiplicity is taken into account. We note that the spectral problem considered in Sect. \ref{SJ24} represents the fourth problem in this list of spectral problems, where no restrictions for $X=0$ and $X=\Lambda/2$ are given.

\begin{proposition}\label{P22} The following properties are valid:
\begin{equation}\label{F1a}
\nu^{0*}_0<0,\;\;\nu^{*0}_0<0,\;\;\nu^{00}_0=0
\end{equation}
and
\begin{equation}\label{F1aa}
\nu^{0*}_1>\mu_1,\;\;\nu^{*0}_1>\mu_1,\;\;\nu^{00}_1>0.
\end{equation}
\end{proposition}
\begin{proof} The function $v_*$ from Proposition \ref{Pr1} satisfies the problem (\ref{Ja24ab}) with $\nu^{00}=0$ and it is positive inside $D$. Therefore $\nu^{00}$ is the lowest eigenvalue of (\ref{Ja24ab}) and we arrive at the last relation in (\ref{F1a}). Since
\begin{equation}\label{F1b}
H^1_{000}(D)\subset H^1_{00*}(D)\;\;\mbox{and}\;\;H^1_{000}(D)\subset H^1_{0*0}(D),
\end{equation}
we obtain the first two inequalities in (\ref{F1a}), compare with the proof of Proposition \ref{Pr1}.

 Inclusions
\begin{equation}\label{F1ba}
H^1_{00*}(D)\subset H^1_{0}(D)\;\;\mbox{and}\;\;H^1_{0*0}(D)\subset H^1_{0}(D)
\end{equation}
together with (\ref{F1b}) implies the inequalities (\ref{F1aa}).
\end{proof}

\section{Generalized eigenvalue problem}\label{S26a}

To study subharmonic and more general bifurcations  a more general eigenvalue problem, which includes quasi momentum as a parameter, is needed. This section is devoted to this spectral problem.

As it is known (see, for example \cite{KT}, \cite{Na})  all bounded solutions to the problem
\begin{eqnarray}\label{Sept11a}
&&\Delta w+\omega'(\Psi)w+\mu w=0 \;\;\mbox{in ${\mathcal D}_\xi$},\nonumber\\
&&\partial_\nu w-\rho w=0\;\;\mbox{on ${\mathcal S}_\xi$},\nonumber\\
&&w=0\;\;\mbox{for $y=0$},
\end{eqnarray}
with real $\mu$, are described by
$$
w(X,Y)=\sum_ja_je^{i\tau_jX}w_j(X,Y),
$$
where $\tau_j\in [0,\tau_*)$ and $w_j$ is a $\Lambda$-periodic solution to the problem
\begin{eqnarray}\label{Sept7a}
&&(\partial_X+i\tau)^2w+\partial_Y^2w+\omega'(\Psi)w+\mu w=0 \;\;\mbox{in ${\mathcal D}_\xi$},\nonumber\\
&&N(\tau) w-\rho w=0\;\;\mbox{on ${\mathcal S}_\xi$},\nonumber\\
&&w=0\;\;\mbox{for $Y=0$},
\end{eqnarray}
where
$$
N(\tau)w=e^{-i\tau X}\partial_\nu (e^{i\tau X}w).
$$
The number $\tau$ is called quasi momentum. It is not assumed in this consideration that $w$ is even.
By \cite{ShaSo} the number of such $\tau_j$  is finite for a fixed $\mu$.
In this section we prove important estimates for eigenvalues of the generalized eigenvalue problem.

\subsection{Variational formulation}

The problem (\ref{Sept7a}) is a spectral problem for a self-adjoint operator for every real $\tau$.
To give its variational formulation we introduce
$$
\Omega=\{(X,Y)\,:\,-\Lambda/2<X<\Lambda/2,\;0<Y<\xi(X)\}
$$
and denote by $H^1_{0,p}(\Omega)$ the subspace of all function $u$ in $H^1(\Omega)$, which satisfies $u(-\Lambda/2,Y)=u(\Lambda/2,Y)$ for $Y\in (0,\xi(\Lambda/2)$ and $u(X,0)=0$ for $X\in (-\Lambda/2,\Lambda/2)$.
We put
\begin{equation}\label{Sept7aa}
{\bf a}(u,v;\tau)=\int_{\Omega}\Big((\partial_X+i\tau) u\overline{(\partial_X+i\tau)v}+\partial_Yu\partial_Y\overline{v}-\omega'(\Psi)u\overline{v}\Big)dXdY-\int_{-\Lambda/2}^{\Lambda/2}\rho(X)u\overline{v}dS,
\end{equation}
which can be written also as
\begin{equation*}
{\bf a}(u,v;\tau)={\bf a}(u.v)+i\tau {\bf b}(u,v)+\tau^2{\bf c}(u,v),
\end{equation*}
where
\begin{equation*}
{\bf a}(u,v)=a_\Omega(u,v),\;\;{\bf b}(u,v)=\int_\Omega (u\partial_X\overline{v}-\partial_Xu\overline{v})dXdY,\;\;{\bf c}(u,v)=(u,v)_{L^2(\Omega)}.
\end{equation*}
Consider the spectral problem: find $w\in H^1_{0p}(\Omega)$ and $\mu\in\Bbb R$ satisfying
\begin{equation}\label{Nov1b}
{\bf a}(w,V;\tau)=\widehat{\mu}(\tau)(\phi,V)_{L^2(\Omega)}\;\;\mbox{for all $V\in H^1_{0,p}(\Omega)$}.
\end{equation}
This is a variational formulation of the spectral problem (\ref{Sept7a}).
Note that we do not suppose in this formulation the functions $w$ and $V$ are even or odd, both of them are only periodic.
We will consider the form ${\bf a}$ only for real $\tau$. In this case the form is symmetric and the spectrum consists of isolated real eigenvalues of finite multiplicity.

Denote by $\widehat{\mu}_j(\tau)$, $\tau\in\Bbb R$, the eigenvalues of the problem
 (\ref{Sept7a})  (equivalently (\ref{Nov1b})) numerated according to the increasing order
 $$
 \widehat{\mu}_0(\tau)\leq \widehat{\mu}_1(\tau)\leq\cdots
 $$
 Clearly
 \begin{equation*}
 \widehat{\mu}_0(0)=\mu_0,\;\;\widehat{\mu}_1(0)=\min (\mu_1,0),\;\;\widehat{\mu}_2(0)=\min(\max (\mu_1,0),\mu_2,\nu^{00}_1)
 \end{equation*}
 and
 \begin{equation}\label{M16a}
 \widehat{\mu}_3(0)>0\;\;\mbox{if $\mu_2>0$}.
 \end{equation}
   Here we use that the first and second eigenvalues of the problem (\ref{J17ax}), considered for $\Lambda$-periodic odd functions coincide with $\nu^{00}_0$ and $\nu^{00}_1$ respectively.

There is another variational formulation. We put  $w=e^{-i\tau X}\phi$ in (\ref{Nov1b}) and $V=e^{-i\tau X}v$, $v\in H^1_0(\Omega,\tau)$, where
$$
H^1_0(\Omega,\tau)=\{v\in H^1_0(\Omega)\,:\,v(-\Lambda/2,Y)=e^{-i\tau\Lambda}v(\Lambda/2,Y)\}.
$$
Then the spectral problem (\ref{Nov1b}) is equivalent to: find $\widehat{\mu}(\tau)$ and $\phi\in H^1_{0}(\Omega,\tau)$ satisfying
\begin{equation}\label{Nov1bb}
a(\phi,v)=\widehat{\mu}(\tau)(\phi,v)_{L^2(\Omega)}\;\;\mbox{for all $v\in H^1_{0}(\Omega,\tau)$}.
\end{equation}

\subsection{Estimates of the eigenvalues $\widehat{\mu}_j$.}

We will need some estimate for the function $\widehat{\mu}_j(\tau)$. For this purpose we introduce 
 two auxiliary spectral problems. Let $H^1_{00}(\Omega)$ consists of functions from $H^1(\Omega)$ which vanish for $Y=0$ and $x=\pm\Lambda/2$.

The first spectral problem is to find $\mu_D\in [0,\tau_*)$ and $\phi\in H^1_{00}(\Omega)$ such that
\begin{equation}\label{J6b}
a(\phi,v)=\mu_D (\phi,v)_{L^2(\Omega)}\;\;\mbox{for all $v\in H^1_{00}(\Omega)$}.
\end{equation}
The second one is to find $\mu_N\in [0,\tau_*)$ and $\phi\in H^1_{0}(\Omega)$ such that
\begin{equation}\label{J6bc}
a(\phi,v)=\mu_N (\phi,v)_{L^2(\Omega)}\;\;\mbox{for all $v\in H^1_{0}(\Omega)$}.
\end{equation}
We denote by $\{\mu_{Dj}\}$ and $\{\mu_{Nj}\}$, $j=0,1,\ldots$, the eigenvalues of the problems (\ref{J6b}) and (\ref{J6bc}) respectively numerated according to
$$
\mu_{D0}\leq \mu_{D1}\leq\cdots,\;\;\mu_{N0}\leq \mu_{N1}\leq\cdots,
$$
where the multiplicity is taken into account.
Since
\begin{equation}\label{Ja16a}
H^1_{00}(\Omega)\subset H^1_{0}(\Omega,\tau)\subset H^1_{0}(\Omega),
\end{equation}
we conclude
\begin{equation}\label{Ja11a}
\mu_{Nj}\leq\widehat{\mu}_j(\tau)\leq\mu_{Dj},\;\;j=0,1,\ldots
\end{equation}
Note that both eigenvalues $\mu_{Dj}$ and $\mu_{Nj}$ do not depend on $\tau$.

In the following lemma we show that all inequalities (\ref{Ja11a}) are strong

\begin{lemma}\label{L25} Let $\xi$ be not identically constant. Then the following inequalities hold:
\begin{equation}\label{Ja14aa}
\mu_{Nj}< \widehat{\mu}_j(\tau)< \mu_{Dj},\;\;j=0,1,\ldots
\end{equation}
for $\tau\in (0,\tau_*)$.
\end{lemma}
\begin{proof} First, we assume  that $n$ satisfies
\begin{equation}\label{Ja16aa}
\mu_{Dn}>\mu_{D(n-1)}
\end{equation}
and prove the  inequalities in right-hand side of (\ref{Ja14aa}). 
Let $\Phi_j\in H^1_{00}(\Omega)$ be the eigenfunction of (\ref{J6b}) corresponding to $\mu_{Dj}$, $j=0,\ldots$, and let $X_n$ be the space of linear combinations of $\{\Phi_j\}_{j=0}^{n-1}$ and
$$
{\mathcal X}_n=\{\Phi\in H^1_{00}(\Omega)\,:\,(\Phi,\Phi_j)_{L^2(\Omega)},\,j=0,\ldots,n-1\}.
$$
Then
$$
\mu_{Dn}=\min_{\Phi\in X_n}\frac{a(\Phi,\Phi)}{||\Phi||_{L^2(\Omega)}^2}.
$$
We assume here and in what follows that $\Phi\neq 0$  in similar relations.

Introduce the subspace
$$
{\mathcal Y}_n={\mathcal Y}_n(\tau)=\{\Phi\in H^1_{0}(\Omega,\tau)\,:\,a(\Phi,\Phi_j)-\mu(\Phi,\Phi_j)=0,\,j=0,\ldots,n-1\},
$$
where we use  the short notation $\mu=\mu_{Dn}$.
The relation ${\mathcal X}_n\subset{\mathcal Y}_n$, follows from
$$
a(\Phi,\Phi_j)-\mu (\Phi,\Phi_j)_{L^2(\Omega)}=(\mu_{Dj}-\mu) (\Phi,\Phi_j)_{L^2(\Omega)}=0\;\;\mbox{for $\Phi\in{\mathcal X}_n$, $j=0,\ldots,n-1$}.
$$
 To show that the codimension of ${\mathcal Y}_n$ is $n$, we assume that there exists
$$
\Phi_*=\sum_{j=0}^{n-1}c_j\Phi_j
$$
such that
$$
a(\Phi,\Phi_*)-\mu(\Phi,\Phi_*)_{L^2(\Omega)}=0\;\; \mbox{for all $\Phi\in H^1_0(\Omega,\tau)$}.
$$
Due to (\ref{Ja16aa})
$$
a(\Phi_*,\Phi_*)-\mu ||\Phi_*||^2_{L^2(\Omega)}<0
$$
 if one of coefficient $c_j$ is non-zero and hence co-dimension of ${\mathcal Y}_n$ is $n$. Let us also prove that every $u\in H^1_0(\Omega,\tau)$ admits the representation $u=\Phi+\phi$, where $\Phi\in {\mathcal Y}_n$ and $\phi\in X_n$. Indeed, we choose $c_j$ and $\phi$ according to
 $$
 c_j(\mu_{Dj}-\mu)||\Phi_j||^2_{L^2(\Omega)}=a(u,\Phi_j)-\mu(u,\Phi_j)_{L^2(\Omega)},\;\;\phi=\sum_{j=0}^{n-1}c_j\Phi_j.
 $$
 Then
 $$
 a(\Phi,\Phi_j)-\mu(\Phi,\Phi_j)_{L^2(\Omega)}=a(u,\Phi_j)-\mu(u,\Phi_j)_{L^2(\Omega)}+a(\phi,\Phi_j)+\mu(\phi,\Phi_j)_{L^2(\Omega)}=0.
 $$
 Therefore $\Phi\in {\mathcal Y}_n$.

By the  min-max principle for eigenvalues of (\ref{Nov1bb}):
\begin{equation}\label{Ja15a}
\widehat{\mu}_{n}\leq \min_{\Phi\in {\mathcal Y}_n}\frac{a(\Phi,\Phi)}{||\Phi||_{L^2(\Omega)}^2},
\end{equation}
If this inequality is strong then the right inequality in (\ref{Ja14aa}) is proved. Assume that
\begin{equation}\label{Ja15aa}
\widehat{\mu}_{n}= \min_{\Phi\in {\mathcal Y}_n}\frac{a(\Phi,\Phi)}{||\Phi||_{L^2(\Omega)}^2}.
\end{equation}
Then the minimum is attained at a certain function $\widehat{\Phi}\in {\mathcal Y}_n$ and
\begin{equation}\label{Ja15b}
a(\widehat{\Phi},v)=\mu (\widehat{\Phi},v)_{L^2(\Omega)}\;\;\mbox{for all $v\in {\mathcal Y}_n$}.
\end{equation}
Using the definition of ${\mathcal Y}_n$, we  conclude that
\begin{equation*}
a(\widehat{\Phi},v+g)-\mu (\widehat{\Phi},v+g)_{L^2(\Omega)}=0\;\;\mbox{for all $g\in X_n$}.
\end{equation*}
 Therefore $\widehat{\Phi}\in {\mathcal Y}_n$ is an eigenfunction of (\ref{Nov1bb}) and (\ref{J6b}) with $\mu=\mu_{Dn}$ considered in the space $H^1_0(\Omega,\tau)$ and $H^1_{00}(\Omega)$.
Considering the functions $\widehat{\Phi}(X,Y)+\widehat{\Phi}(-X,Y)$ and $\widehat{\Phi}(X,Y)-\widehat{\Phi}(-X,Y)$, which are even and odd respectively and using that $\tau\in (0,\tau_*)$ (this implies that $\partial_X\Phi=0$ for $X=\pm\Lambda/2$), we conclude that $\widehat{\Phi}=0$, since there is a homogeneous Cauchy data at $X=\pm\Lambda/2$ for the function $\widehat{\Phi}$.  Thus the right inequality in (\ref{Ja14aa}) is strong.

Consider the case when (\ref{Ja16aa}) is not valid. Then we choose $l$ such that
$$
\mu_{Dn}=\cdot=\mu_{D(n-l)}>\mu_{D(n-l-1)}.
$$
By just proved, $\widehat{\mu}_{n-l}(\tau)<\mu_{D(n-l)}$, which implies $\widehat{\mu}_{n}(\tau)<\mu_{D(n)}$.

Let us turn to the left inequality in (\ref{Ja14aa}).
Now, we assume  that $n$ satisfies
\begin{equation}\label{Ja19aa}
\widehat{\mu}_{n}>\widehat{\mu}_{n-1}.
\end{equation}
We represent $H^1_0(\Omega)$ as
$$
H^1_0(\Omega)={\mathcal X}+{\mathcal Y},
$$
where ${\mathcal X}$ consists of even functions in $H^1_0(\Omega)$ and ${\mathcal Y}$ consists of odd functions in $H^1_0(\Omega)$.

Let $\Phi_j\in H^1_{0}(\Omega,\tau)$ be the eigenfunction of (\ref{J6b}) corresponding to $\mu_{Dj}$, $j=0,\ldots$, and let now $X_n$ be the space of linear combinations of $\{\Phi_j\}_{j=0}^{n-1}$ and
$$
{\mathcal X}_n(\tau)=\{\Phi\in H^1_{0}(\Omega,\tau)\,:\,(\Phi,\Phi_j)_{L^2(\Omega)},\,j=0,\ldots,n-1\}.
$$
Then
$$
\widehat{\mu}_{n}=\min_{\Phi\in X_n}\frac{a(\Phi,\Phi)}{||\Phi||_{L^2(\Omega)}^2}.
$$

Introduce the subspace
$$
{\mathcal Y}_n=\{\Phi\in H^1_{0}(\Omega)\,:\,a(\Phi,\Phi_j)-\mu(\Phi,\Phi_j)=0,\,j=0,\ldots,n-1\},
$$
where we use  the short notation $\mu=\mu_{Dn}$.
The relation ${\mathcal X}_n\subset{\mathcal Y}_n$, follows from
$$
a(\Phi,\Phi_j)-\mu (\Phi,\Phi_j)_{L^2(\Omega)}=(\mu_{Dj}-\mu) (\Phi,\Phi_j)_{L^2(\Omega)}=0\;\;\mbox{for $\Phi\in{\mathcal X}_n$, $j=0,\ldots,n-1$}.
$$
In the same way as before one can show that the codimension of ${\mathcal Y}_n$ is equal to $n$ and
that every $u\in H^2_0(\Omega,\tau)$ admits the representation $u=\Phi+\phi$, where $\Phi\in {\mathcal Y}_n$ and $\phi\in X_n$.

 By the  min-max principle for eigenvalues of (\ref{Nov1bb}):
\begin{equation}\label{Ja15ax}
\widehat{\mu}_{n}\leq \min_{\Phi\in {\mathcal Y}_n}\frac{a(\Phi,\Phi)}{||\Phi||_{L^2(\Omega)}^2},
\end{equation}
If this inequality is strong then the right inequality in (\ref{Ja14aa}) is proved. Assume that
\begin{equation}\label{Ja15aax}
\widehat{\mu}_{n}= \min_{\Phi\in {\mathcal Y}_n}\frac{a(\Phi,\Phi)}{||\Phi||_{L^2(\Omega)}^2}.
\end{equation}
Then the minimum is attained at a certain function $\widehat{\Phi}\in {\mathcal Y}_n$ and
\begin{equation*}
a(\widehat{\Phi},v)=\mu (\widehat{\Phi},v)_{L^2(\Omega)}\;\;\mbox{for all $v\in {\mathcal Y}_n$}.
\end{equation*}
Using the definition of ${\mathcal Y}_n$, we  conclude that
\begin{equation*}
a(\widehat{\Phi},v+g)-\mu (\widehat{\Phi},v+g)_{L^2(\Omega)}=0\;\;\mbox{for all $g\in X_n$}.
\end{equation*}
 Therefore $\widehat{\Phi}\in {\mathcal Y}_n$ is an eigenfunction of (\ref{Nov1bb}) and (\ref{J6b}) with $\mu=\widehat{\mu}_n$ considered in the space $H^1_0(\Omega)$ and $H^1_{0}(\Omega,\tau)$.
  The same argument as before proves that the right inequality in (\ref{Ja14aa}) is strong.

\end{proof}

The next lemma gives important estimates for $\widehat{\mu}_j$, which are essential for the proof our main theorem.

\begin{lemma}\label{L26a} For all $\tau\in (0,\tau_*)$

{\rm (i)}
\begin{equation}\label{Ja19a}
\mu_0<\widehat{\mu}_0(\tau)<\nu^{*0}_0\;\;.
\end{equation}

{\rm (ii)} If $0<\nu^{*0}_1$ and $\nu^{0*}_0<\mu_1$ then
\begin{equation}\label{J6a}
\nu^{0*}_0<\widehat{\mu}_1(\tau)<0,
\end{equation}
and
\begin{equation}\label{J6aa}
\mu_1<\widehat{\mu}_2(\tau)<\nu^{*0}_1 
\end{equation}
{\rm (iii)}
\begin{equation}\label{Ja25a}
\nu^{0*}_1<\widehat{\mu}_3(\tau).
\end{equation}
\end{lemma}
\begin{proof}
First let us prove the equalities
\begin{equation}\label{Ja16c}
\mu_{D0}=\nu^{*0}_0,\;\;\mu_{D1}=\nu^{00}_0,\;\;\mu_{D2}=\nu^{*0}_1
\end{equation}
assuming $\nu^{00}_0<\nu^{*0}_1$ (this is needed for the second and third relations above).

 We represent $H^1_{00}(\Omega)$ as
$$
H^1_{00}(\Omega)=\widehat{X}+\widehat{Y},
$$
where $\widehat{X}$ consists of even functions in $H^1_{00}(\Omega)$ and $\widehat{Y}$ consists of odd functions in $H^1_{00}(\Omega)$. Then the spaces $\widehat{X}$ and $\widehat{Y}$ are invariant for this spectral problem.


Introduce two spectral problems
\begin{equation}\label{Ja16b}
a_\Omega(u,v)=\mu(u,v)_{L^2(\Omega)}\;\;\mbox{for all $v\in \widehat{X}$},
\end{equation}
where $u\in \widehat{X}$, and
\begin{equation}\label{Ja16ba}
a_\Omega(u,v)=\mu(u,v)_{L^2(\Omega)}\;\;\mbox{for all $v\in \widehat{Y}$},
\end{equation}
where $u\in \widehat{Y}$. Then the eigenvalues and eigenfunctions of the problems (\ref{Ja16b}) and (\ref{Ja16ba}) coincide with the  eigenvalues and eigenfunctions of the problem (\ref{J6b}). Furthermore, the eigenvalues of the problem (\ref{Ja16b}) coincides with the eigenvalues of the problem (\ref{Ja24a}) and the eigenvalues of the problem (\ref{Ja16ba}) coincides with the eigenvalues of the problem (\ref{Ja24ab}). This implies (\ref{Ja16c}) and due to Proposition \ref{P22} and Lemma \ref{L25} we get the right-hand inequalities in (\ref{Ja19a})-(\ref{J6aa}).

Let us turn to the estimates (\ref{Ja19a})-(\ref{J6aa}) and (\ref{Ja25a}) from below.
We start from proving
\begin{equation}\label{Ja16ca}
\mu_{N0}=\mu_0,\;\;\mu_{N1}=\nu^{0*}_0,\;\;\mu_{N2}=\mu_1,\;\;\mu_{N3}=\nu^{0*}_1
\end{equation}
assuming $\nu^{0*}_0<\mu_1$ (this is needed only for the second and third relations above).

 We use the representation
$$
H^1_0(\Omega)=\tilde{X}+\tilde{Y},
$$
where $\tilde{X}$ consists of even functions with respect to $X$ in $H^1_{0}(\Omega)$ and $\tilde{Y}$ consists of odd functions in $H^1_{0}(\Omega)$.
Introduce two more spectral problem
\begin{equation}\label{Ja16bv}
a(u,v)=\mu(u,v)_{L^2(\Omega)}\;\;\mbox{for all $v\in \tilde{X}$},
\end{equation}
where $u\in \tilde{X}$, and
\begin{equation}\label{Ja16bav}
a(u,v)=\mu(u,v)_{L^2(\Omega)}\;\;\mbox{for all $v\in \tilde{Y}$},
\end{equation}
where $u\in \tilde{Y}$. Then the eigenvalues and eigenfunctions of the problems (\ref{Ja16bv}) and (\ref{Ja16bav}) coincides with of the  eigenvalues and eigenfunctions of the problem (\ref{J6bc}).
Furthermore, the eigenvalues of the problem (\ref{Ja16bv}) coincides with the eigenvalues of the problem (\ref{Sept8a}) and the eigenvalues of the problem (\ref{Ja16bav}) coincides with the eigenvalues of the problem (\ref{Ja24aa}). This implies  (\ref{Ja16ca}) and due to Proposition \ref{P22} and Lemma \ref{L25} we get the right-hand inequalities in (\ref{Ja19a})-(\ref{Ja25a}).


\end{proof}

\begin{corollary}\label{C15a} If $\mu_1>0$ then
$$
\widehat{\mu}_1(\tau)<0,\;\;\widehat{\mu}_2(\tau)>0\;\;\mbox{for $\tau\in (0,\tau_*)$ and}\;\;\widehat{\mu}_2(0)=\mu_1>0,\;\widehat{\mu}_1=0.
$$
Moreover the eigenvalue $\mu=0$ is simple with the eigenfunction $\psi_x$.
\end{corollary}

\subsection{Asymptotics of eigenvalues $\widehat{\mu}_1(\tau)$ and $\widehat{\mu}_2(\tau)$}

In this section we assume that $\mu_1=0$. According to Proposition \ref{PF10a} this eigenvalue is simple in the space of periodic even function.  Since $0$ is always simple eigenvalue in the space of odd periodic functions, the multiplicity of the eigenvalue $\mu_1$ is two in the space of periodic functions. One can check that
if $\mu$ is an eigenvalue of (\ref{Sept7a}) for a certain $\tau$ then $\mu$ is also the eigenvalue of the same multiplicity
of the same problem for $-\tau$.
Therefore
$$
\widehat{\mu}_1(-\tau)=\widehat{\mu}_1(\tau),\;\;\widehat{\mu}_2(-\tau)=\widehat{\mu}_2(\tau).
$$
Furthermore due to Lemma \ref{L26a}
\begin{equation}\label{F3a}
\widehat{\mu}_1(\tau)<0\;\;\mbox{and}\;\;\widehat{\mu}_2(\tau)>0\;\;\mbox{for $\tau\in(0,\tau_*)$}.
\end{equation}
Since $\mu(-\tau)$, $\mu(\tau+\tau_*)$ are eigenvalues if $\mu(\tau)$ is an eigenvalue we have
$$
\widehat{\mu}_1(\tau_*-\tau)=\widehat{\mu}_1(\tau)\;\;\mbox{and}\;\;\widehat{\mu}_2(\tau_*-\tau)=\widehat{\mu}_2(\tau).
$$

Moreover if $\phi(X,Y;\tau)$ is an eigenfunction corresponding to $\mu(\tau)$ then  $\phi(-X.Y;-\tau)$ is an eigenfunction corresponding to the eigenvalue $\mu(-\tau)$.

\begin{lemma}\label{LD2a} Assume that $\mu_1=0$. 
Then one of the following two options is valid

{\rm (i)}  There exists an integer $n\geq 1$ such that
 \begin{equation}\label{Okt5a}
\widehat{\mu}_1(\tau)=-\kappa_n\tau^{2n-1}+O(\tau^{2n}),\;\;\widehat{\mu}_2(\tau)=\kappa_n\tau^{2n-1}+O(\tau^{2n})
\end{equation}
for small positive $\tau$, where $\kappa_n> 0$.

{\rm (ii)} There exist  integers $n,\,m\geq 1$ such that
\begin{equation}\label{F3b}
\widehat{\mu}_1(\tau)=A\tau^{2n}+O(|\tau|^{2n+1}),\;\;\widehat{\mu}_2(\tau)=B\tau^{2m}+O(|\tau|^{2m+1})
\end{equation}
for small positive $\tau$, where $A<0$ and $B>0$.

\end{lemma}

\begin{proof}  Since the multiplicity of the  eigenvalue $\mu=0$ of the problem (\ref{Sept7a}) has multiplicity two for $\tau=0$ there are two analytic in $\tau$ branches of eigenvalues of (\ref{Sept7a}), denote them by $\theta_1(\tau)$ and $\theta_2(\tau)$ such that $\theta_1(0)=\theta_2(0)=0$.
Since $\theta_1(-\tau)$ and $\theta_2(-\tau)$ are also eigenvalues one of the following options is valid:

(a) $\theta_1(\tau)=\theta_2(\tau)$ and the function $\theta_1$ is even with respect to $\tau$;

(b)  $\theta_1(-\tau)=\theta_2(\tau)$;

(c) $\theta_1(-\tau)=\theta_1(\tau)$ and $\theta_2(-\tau)=\theta_2(\tau)$.

Observing that the functions $\theta_j(\tau)$ coincides with one of functions $\widehat{\mu}_j$, $j=1,2$, for small positive $\tau$, we conclude that only the options (b) and (c) may occur. This leads to the options (i) and (ii) in our lemma. The sign of $\kappa_n$, $A$ and $B$ follows from (\ref{F3a}).
\end{proof}

The asymptotics of the corresponding eigenfunctions is given in the following

\begin{lemma} Let $\upsilon_1(\tau)$ and $\upsilon_2(\tau)$ be eigenfunctions analytically depending on $\tau$  corresponding to the eigenvalues $\theta_1$ and $\theta_2$ in the proof of {\rm Lemma \ref{LD2a}}.
In the case {\rm (\ref{Okt5a})} the
  corresponding eigenfunctions satisfy
\begin{equation}\label{Okt5aa}
\upsilon_2(0)=\phi_0+i\beta\psi_0,\;\;\upsilon_1(0)=\phi_0-i\beta^{-1}\psi_0,\;\;\;\mbox{with $\beta\neq 0$}.
\end{equation}
In the case {\rm (\ref{F3b})}
the corresponding eigenfunctions satisfy
\begin{equation}\label{Okt5aaz}
\upsilon_2(0)=\phi_0\;\;\mbox{and}\;\;\upsilon_1(0)=\psi_0.
\end{equation}
\end{lemma}
\begin{proof}
 Introduce polynomials
$$
\Phi(\tau)=\phi_0+i\tau\phi_1+\cdots+(i\tau)^{p}\phi_{p}\;\;\mbox{and}\;\;\Psi(\tau)=\psi_0+i\tau\psi_1+\cdots+(i\tau)^{p}\psi_{p},
$$
where $\phi_j$, $\psi_k$ are real-valued functions from $H^1_{p0}(\Omega)$ such that they are even when $j$ is even and $k$ is odd, and they are odd when $j$ is odd and $k$ is even. Let us construct the functions $\phi_j$ and $\psi_k$, $j,k=1,\ldots,p-1$, satisfying
$$
{\bf a}(\Phi(\tau),v;\tau)=O(|\tau|^p),\;\;{\bf a}(\Psi(\tau),v;\tau)=O(|\tau|^p)\;\;\mbox{for all $v\in H^1_{p0}(\Omega)$}.
$$
Equating terms of the same power of $\tau$ in the first relation, we obtain
\begin{equation}\label{Okt23a}
{\bf a}(\phi_1,v)+{\bf b}(\phi_0,v)=0
\end{equation}
and
\begin{equation}\label{Okt23aa}
{\bf a}(\phi_{k+2},v)+{\bf b}(\phi_{k+1},v)+{\bf c}(\phi_{k},v)=0\;\;\mbox{for $k=0,1,\ldots,p-3$ and $v\in H^1_0(\Omega)$}.
\end{equation}
Similar relations hold with $\phi$ replaced by $\psi$. Equations (\ref{Okt23a}) and (\ref{Okt23aa}) are solvable if "the right-hand side is orthogonal to the kernel of the main operator corresponding to the form ${\bf a}$", i.e.
\begin{equation}\label{Okt23b}
{\bf b}(\phi_0,\psi_0)=0,
\end{equation}
\begin{equation}\label{Okt23ba}
{\bf b}(\phi_{k+1},\phi_0)+{\bf c}(\phi_{k},\phi_0)=0,\;\;\;{\bf b}(\phi_{k+1},\psi_0)+{\bf c}(\phi_{k},\psi_0)=0,
\end{equation}
and
\begin{equation}\label{D5b}
{\bf b}(\psi_{k+1},\phi_0)+{\bf c}(\psi_{k},\phi_0)=0,\;\;\;{\bf b}(\psi_{k+1},\psi_0)+{\bf c}(\psi_{k},\psi_0)=0.
\end{equation}
Here we used that the form ${\bf b}$ is anti-symmetric.
In this case solutions are not unique and we choose the solutions orthogonal to the kernel, i.e.
\begin{equation}\label{Okt24a}
{\bf c}(\phi_k,\phi_0)=0,\;\;{\bf c}(\phi_k,\psi_0)=0,\;\;{\bf c}(\psi_k,\phi_0)=0,\;\;{\bf c}(\psi_k,\psi_0)=0,
\end{equation}
for $k=1,\ldots,p-1$. We note that some terms in relations (\ref{Okt24a}) vanish since some of functions are odd and some of them are even.

(i) Assume that (\ref{Okt5a}) is valid and $p=2n-1$. In order to write equations similar to (\ref{Okt23a}) and (\ref{Okt23aa}) for the next term we introduce the function
$$
\Phi(\alpha,\beta,\tau)=\alpha\Phi(\tau)+\beta\Psi(\tau),
$$
where $\alpha$ and $\beta$ are unknown constants. The equation for finding $\alpha $ and $\beta$ is
$$
{\bf a}(\Phi(\alpha,\beta,\tau),v;\tau)=\kappa (\alpha \phi_0+\beta\psi_0,v)+O(|\tau|^{p+1})\;\;\mbox{for all $v\in H^1_{p0}(\Omega$)}.
$$
as the result we obtain
$$
{\bf a}(\alpha\phi_p+\beta\psi_p,v)+{\bf b}(\alpha\phi_{p-1}+\beta\psi_{p-1},v)+{\bf c}(\alpha\phi_{p-2}+\beta\psi_{p-2},v)=(-i)^p\kappa{\bf c}(\alpha\phi_0+\beta\psi_0,v)
$$
for all $v\in H^1_{p0}(\Omega)$.
This problem is solvable if
\begin{equation}\label{D1b}
{\bf b}(\alpha\phi_{p-1}+\beta\psi_{p-1},\phi_0)+{\bf c}(\alpha\phi_{p-2}+\beta\psi_{p-2},\phi_0)=(-i)^p\kappa\alpha
\end{equation}
and
\begin{equation}\label{D1ba}
{\bf b}(\alpha\phi_{p-1}+\beta\psi_{p-1},\psi_0)+{\bf c}(\alpha\phi_{p-2}+\beta\psi_{p-2},\psi_0)=(-i)^p\kappa\beta,
\end{equation}

From (\ref{D1b}) and (\ref{D1ba}) it follows
\begin{equation}\label{D2a}
\beta{\bf b}(\psi_{p-1},\phi_0)=(-i)^p\kappa\alpha
\end{equation}
and
\begin{equation}\label{D2aa}
\alpha{\bf b}(\phi_{p-1},\psi_0)=(-i)^p\kappa\beta,
\end{equation}
which implies
\begin{equation}\label{D2b}
\kappa^2=(-1)^p{\bf b}(\phi_{p-1},\psi_0){\bf b}(\psi_{p-1},\phi_0).
\end{equation}
This implies that the left-hand side in (\ref{D2b}) is positive, $\kappa$ and $-\kappa$  satisfy and
$$
\alpha=1\;\;\mbox{and}\;\;\beta=(-i)^p\kappa^{-1}{\bf b}(\phi_{p-1},\psi_0).
$$

This implies that
$$
(-1)^p{\bf b}(\phi_{p-1},\psi_0){\bf b}(\phi_{p-1},\psi_0)=\alpha^2.
$$

(ii) Let (\ref{F3b}) be valid.  The same calculation as above shows
$$
\alpha{\bf b}(\phi_{p-1},\phi_0)=(-1)^{p/2}\kappa\alpha,
$$
and
$$
\beta{\bf b}(\psi_{p-1},\psi_0)=(-1)^{p/2}\kappa\beta.
$$
So one of quantities ${\bf b}(\phi_{p-1},\phi_0$ or ${\bf b}(\psi_{p-1},\psi_0)$ must be different from $0$ and we arrive at (\ref{Okt5aaz}).


\end{proof}





\begin{lemma}\label{L16s} Let the eigenvalue $\mu=0$ is simple with the eigenfunction $\psi_x$. Then
$$
\mu(\tau)=c\tau^2+O(\tau^4),
$$
where $c$ is negative.
\end{lemma}
\begin{proof}
We are looking for the eigenvalue and eigenfunction in the form
$$
\mu(\tau)=c_1\tau+c_2\tau^2+\cdots \;\; \mbox{and} \;\; u(\tau)=u_0+i\tau u_1+(i\tau)^2u_2+\cdots
$$
then
$$
{\bf a}(u_0,w)=0\;\;\mbox{for all $w\in H^1_0(\Omega)$}.
$$
Therefore $u_0=\psi_x$. Furthermore,
$$
i\tau{\bf a}(u_1,w)+i\tau{\bf b}(u_0,w)=c_1(u_0,w)_{L^2(\Omega)}\;\;\mbox{for all $w\in H^1_0(\Omega)$}.
$$
Choosing $w=u_0$ we see that $c_1=0$. Taking $w=u_1$, we get
\begin{equation}\label{F11a}
{\bf a}(u_1,u_1)+{\bf b}(u_0,u_1)=0\;\;\Rightarrow\;\;{\bf b}(u_1,u_0)={\bf a}(u_1,u_1).
\end{equation}
Equation for $u_2$ is
$$
{\bf a}(u_2,w)+{\bf b}(u_1,w)+(u_0,w)_{L^2(\Omega)}=-c_2(u_0,w)_{L^2(\Omega)}\;\;\mbox{for all $w\in H^1_0(\Omega)$}.
$$
Taking $w=u_0$ we get
$$
{\bf b}(u_1,u_0)+(u_0,u_0)_{L^2(\Omega)}=-c_2(u_0,u_0)_{L^2(\Omega)}.
$$
Using (\ref{F11a}), we arrive at
$$
c_2||u_0||^2_{L^2(\Omega)}=||u_0||^2_{L^2(\Omega)}+{\bf a}(u_1,u_1),
$$
which proves positivity of $c_2$.

\end{proof}

\section{Subharmonic spectral problems}

 Let $M$ be a positive integer.
As is known (see for example \cite{Na}, \cite{KT}),  $M\Lambda$-periodic, even solutions to the problem (\ref{J17ax}) are linear combinations of functions
\begin{equation}\label{Sept14a}
e^{i\tau_jX}w_j(X,Y)+e^{-i\tau_jX}w_j(-X,Y),
\end{equation}
where only such $\tau=\tau_j=j\tau_*/M$ and $w=w_j$, $j=0,1,\ldots, M-1$, are taken which solves the problem (\ref{Sept7a}) with $\mu=0$.
Therefore the knowledge of the characteristic exponents $\tau_j$ helps to describe $M$-subharmonic solutions.

The spectral problem (\ref{J17ax}) can be considered on the functions $w$ which are even and has period $M\Lambda$. We shall call this spectral problem $M$-subharmonic spectral problem. To give a variational formulation of this spectral problem we introduce
$$
D_M=\{(X,Y)\in {\mathcal D}\,:\,0<X<M\Lambda/2\},\;\;S_M=\{(X,Y)\in {\mathcal S}\,:\,0<X<M\Lambda/2\},
$$
and denote  by $H^1_{0}(D_M)$ the space of functions in $H^1(D_M)$ which vanish for $Y=0$. We put
$$
a_M(u,v)=\int_{D_M}\Big(\nabla u\cdot\nabla\overline{v}-\omega'(\Psi)u\overline{v}\Big)dXdY-\int_{0}^{M\Lambda/2}\rho(X)u\overline{v}dS,
$$
Using that $\partial_Xw(0,Y)=0$ and $\partial_Xw(M\Lambda/2,Y)=0$, we arrive at the following variational formulation for the spectral problem (\ref{J17ax}) on even $M\Lambda$-periodic functions: find a  function $\phi\in H^1_0(D_M)$ and the number $\mu=\mu^{(M)}\in \Bbb R$ satisfying 
\begin{equation}\label{Sept9a}
a_M(\phi,v)=\mu^{(M)}(\phi,v)_{L^2(D_M)}\;\;\mbox{for all $v\in H^1_0(D_M)$}.
\end{equation}
As it is known the pair $(\phi,\mu^{(M)})$ solves the problem (\ref{Sept9a}) if and only if $\mu^{(M)}$ is an eigenvalue of the problem (\ref{J17ax}), considered on even, $M\Lambda$-periodic functions, and $\phi$, after a natural extension, is a corresponding eigenfunction.
We numerate the eigenvalues accounting their multiplicity as
$$
\mu^{(M)}_0\leq \mu^{(M)}_1\leq\cdots ,\;\;\mu^{(M)}_j\to\infty\;\;\mbox{as}\;\;j\to\infty.
$$

Denote the number of eigenvalues of the problem (\ref{Sept9a}) which are less than  zero by $n_0(M)$ and by $n(M)$ if we include zero  eigenvalues, i.e.
\begin{eqnarray*}
&&n_0(M)=\mbox{the number of $j$ such that}\;\;\mu^{(M)}_j<0,\\
&&n(M)=\mbox{the number of $j$ such that}\;\;\mu^{(M)}_j\leq 0.
\end{eqnarray*}

\begin{proposition}\label{D5a} Let $\mu_1=0$. Then for every $M=1,2,\ldots$
\begin{equation}\label{Sept5b}
n_0(M)=M\;\;\mbox{when $M$ is odd and}\;\;n(M)=M+1\;\;\mbox{when $M$ is even}
\end{equation}
and
\begin{equation}\label{Ja21a}
n(M)=M+1\;\;\mbox{when $M$ is odd and}\;\;n(M)=M+2\;\;\mbox{when $M$ is even}.
\end{equation}
\end{proposition}
\begin{proof} Using that
$$
\widehat{\mu}_j(\tau)=\widehat{\mu}_j(-\tau)=\widehat{\mu}_j(\tau_*+\tau)\;\;\mbox{for $j=0,1,2$},
$$
we conclude that $\widehat{\mu}_j(\tau)=\widehat{\mu}_j(\tau_*-\tau)$. Due to the description of even, $M\Lambda$ periodic solution given by (\ref{Sept14a}), the number $n(M)$ is equal to the number of indexes $k$ and $m$ such that
\begin{equation}\label{Ja22a}
\widehat{\mu}_0\Big(\frac{k}{M}\Big)\leq 0,\;\;\widehat{\mu}_1\Big(\frac{m}{M}\Big)\leq 0,
\end{equation}
where $0\leq k,m\leq M/2$. Therefore
$$
n(M)=M+1\;\;\mbox{when $M$ is odd and}\;\;n(M)=M+2\;\;\mbox{when $M$ is even},
$$
which proves (\ref{Ja21a}).

Since the equality in (\ref{Ja22a}) is reached only when  $m=0$, we get (\ref{Sept5b}).
\end{proof}

A necessary condition for $M$ subharmonic bifurcation is the absence of  zero eigenvalues of the Frechet derivative corresponding to $M$-subharmonic problem. The next proposition deals with this problem.

\begin{proposition}\label{P16b} If $\mu_1>0$ then the numbers $\mu^{(M)}_j$ are different from $0$.

\end{proposition}
\begin{proof} The proof of this proposition follows directly from Corollary \ref{C15a} and from description of zero eigenvalues of the Frechet derivative for $M$-subharmonic problem given at the beginning of this section.

\end{proof}

\section{Branches of Stokes waves}

\subsection{Uniform stream solution, dispersion equation}

The uniform stream solution $\Psi=U(y)$ with the constant depth $\eta =d$ and $\lambda=1$,  satisfies the problem
\begin{eqnarray}\label{X1}
&&U^{''}+\omega(U)=0\;\;\mbox{on $(0;d)$},\nonumber\\
&&U(0)=0,\;\;U(d)=1,\nonumber\\
&&\frac{1}{2}U'(d)^2+d=R.
\end{eqnarray}
Let $s=U'(0)$ and $s>s_0:=2\max_{\tau\in [0,1]}\Omega(\tau)$, where
$$
\Omega(\tau)=\int_0^\tau \omega(p)dp.
$$
Then the problem (\ref{X1}) has a solution $(U,d)$ with a strongly monotone function $U$ if
\begin{equation}\label{M6a}
{\mathcal R}(s):=\frac{1}{2}s^2+d(s)-\Omega(1)=R.
\end{equation}
In this case $(U,d)$ is found from the relations
\begin{equation}\label{F22a}
y=\int_0^U\frac{d\tau}{\sqrt{s^2-2\Omega(\tau)}},\;\;d=d(s)=\int_0^1\frac{d\tau}{\sqrt{s^2-2\Omega(\tau)}}.
\end{equation}
The equation (\ref{M6a}) is solvable if $R>R_c$,
\begin{equation}\label{F27a}
R_c=\min_{s\geq s_0}{\mathcal R}(s).
\end{equation}
We denote by $s_c$ the point where the minimum in (\ref{F27a}) is attained.

Existence of small amplitude Stokes waves is determined by the dispersion equation (see, for example, \cite{KN14}). It is defined as follows.
 The strong monotonicity of $U$ guarantees that the problem
\begin{equation}\label{Okt6b}
\gamma^{''}+\omega'(U)\gamma-\tau^2\gamma=0,\;\; \gamma(0,\tau)=0,\;\;\gamma(d,\tau)=1
\end{equation}
has a unique solution $\gamma=\gamma(y,\tau)$ for each $\tau\in\Bbb R$, which is even with respect to $\tau$ and depends analytically on $\tau$.
Introduce the function
\begin{equation}\label{Okt6ba}
\sigma(\tau)=\kappa\gamma'(d,\tau)-\kappa^{-1}+\omega(1),\;\;\kappa=\Psi'(d).
\end{equation}
It depends also analytically on $\tau$ and it is strongly increasing with respect to $\tau>0$. Moreover it is an even function.
The dispersion  equation (see, for example \cite{KN14})  is the following
\begin{equation}\label{Okt6bb}
\sigma(\tau)=0.
\end{equation}
It has a positive solution if
\begin{equation}\label{D17a}
\sigma(0)<0.
\end{equation}
By \cite{KN14} this is equivalent to $s+d'(s)<0$ or what is the same
\begin{equation}\label{D19b}
1<\int_0^d\frac{dy}{U'^2(y)}.
\end{equation}
The left-hand side here is equal to $1/F^2$ where $F$ is the Froude number (see \cite{We}). Therefore (\ref{D19b}) means that $F<1$, which is well-known condition for existence of  water waves of small amplitude.
Another equivalent formulation is given by requirement (see, for example \cite{KN11})
\begin{equation}\label{M3aa}
s\in (s_0,s_c)\;\;\mbox{and satisfies (\ref{M6a})}.
\end{equation}
The existence of such $s$ is guaranteed by $R\in (R_c,R_0)$. Moreover in this case the equation ${\mathcal R}(s)=R$ has exactly two solutions $s_-$ and $s_+$, $s_0<s_+<s_c<s_-$. The corresponding solutions to (\ref{X1}) are given by (\ref{F22a}) and we denote them by
\begin{equation}\label{KKKb}
(U_+,d_+)\;\;\mbox{and}\;\;(U_-,d_-).
\end{equation}


The function $\sigma$ has the following asymptotic representation
$$
\sigma(\tau)=\kappa\tau +O(1)\;\;\mbox{for large $\tau$}
$$
and equation (\ref{Okt6bb}) has a unique positive root, which will be denoted by $\tau_*$. It is connected with $\Lambda_0$ by the relation
\begin{equation}\label{J4aq}
\tau_*=\frac{2\pi}{\Lambda_0}.
\end{equation}
The function $\gamma (y,\tau)$ is positive in $(0,d]$ for $\tau>\tau_*$.

Let
\begin{equation}\label{F28a}
\rho_0=\frac{1+\Psi'(d)\Psi^{''}(d)}{\Psi'(d)^2}.
\end{equation}
We note that
$$
\frac{1+\Psi'(d)\Psi^{''}(d)}{\Psi'(d)^2}=\kappa^{-2}-\frac{\omega(1)}{\kappa}
$$
and hence another form for (\ref{Okt6ba}) is
\begin{equation}\label{M21aa}
\sigma(\tau)=\kappa\gamma'(d,\tau)-\kappa\rho_0.
\end{equation}

The spectral problem (\ref{J17ax}) takes the form
\begin{eqnarray}\label{J7a}
&&\Delta w+\omega'(U)w+\mu w=0 \;\;\mbox{for $X\in\Bbb R$ and $0<Y<d$},\nonumber\\
&&\partial_Y w-\rho_0 w=0\;\;\mbox{for $Y=d$},\nonumber\\
&&w=0\;\;\mbox{for $Y=0$},
\end{eqnarray}
where $w$ is even, $\Lambda=2\pi/\tau_*$-function. One can verify that $\mu_0<0$ and corresponding eigenfunction does not depend on $X$, $\mu_1=0$ and corresponding eigenfunction is $\cos(\tau_*X)\gamma(Y;\tau_*)$ and $\mu_2>0$. For small $t$, $\mu_0<0$ and $\mu_2>0$. We assume that
\begin{equation}\label{Ja7aa}
\mu_1(t)>0\;\;\mbox{for small positive $t$}.
\end{equation}
The validity of this assumption will be discussed in forthcoming paper.

\subsection{Partial hodograph transform}\label{SF22}

In what follows we will study  branches of Stokes waves $(\Psi(X,Y;t),\xi(X;t))$ of period $\Lambda(t)$, $t\in\Bbb R$, started from the uniform stream  at $t=0$.
It is convenient to make the following change of variables
\begin{equation}\label{D11a}
x=\lambda X,\;\;y=Y,\;\;\lambda=\frac{\Lambda(0)}{\Lambda(t)}
\end{equation}
in order to deal with the problem with a fixed period. As the result we get
\begin{eqnarray}\label{Okt6aa}
&&\Big(\lambda^2\partial_x^2+\partial_y^2\Big)\psi+\omega(\psi)=0\;\;\mbox{in $D_\eta$},\nonumber\\
&&\frac{1}{2}\Big(\lambda^2\psi_x^2+\psi_y^2\Big)+\eta=R\;\;\mbox{on $B_\eta$},\nonumber\\
&&\psi=1\;\;\mbox{on $B_\eta$},\nonumber\\
&&\psi=0\;\;\mbox{for $y=0$},
\end{eqnarray}
where
$$
\psi(x,y;t)=\Psi(\lambda^{-1}x,y;t)\;\;\mbox{and}\;\;\eta(x;t)=\xi(\lambda^{-1} x;t).
$$
Here all functions have the  same period $\Lambda_0:=\Lambda(0)$, $D_\eta$ and $B_\eta$  are the domain and the free surface  after the change of variables (\ref{D11a}).

Due to the assumption (\ref{J27ba}) we can introduce the variables
$$
q=x,\;\;p=\psi.
$$
Then
$$
q_x=1,\;\;q_y=0,\;\;p_x=\psi_x,\;\;p_y=\psi_y,
$$
and
\begin{equation}\label{F28b}
\psi_x=-\frac{h_q}{h_p},\;\;\psi_y=\frac{1}{h_p},\;\;dxdy=h_pdqdp.
\end{equation}

System (\ref{Okt6aa}) in the new variables takes the form
\begin{eqnarray}\label{J4a}
&&\Big(\frac{1+\lambda^2h_q^2}{2h_p^2}+\Omega(p)\Big)_p-\lambda^2\Big(\frac{h_q}{h_p}\Big)_q=0\;\;\mbox{in $Q$},\nonumber\\
&&\frac{1+\lambda^2h_q^2}{2h_p^2}+h=R\;\;\mbox{for $p=1$},\nonumber\\
&&h=0\;\;\mbox{for $p=0$}.
\end{eqnarray}
Here
$$
Q=\{(q,p)\,:\,q\in\Bbb R\,,\;\;p\in (0,1)\}.
$$
The uniform stream solution corresponding to the solution $U$ of (\ref{X1}) is
\begin{equation}\label{M4c}
H(p)=\int_0^p\frac{d\tau}{\sqrt{s^2-2\Omega(\tau)}},\;\;s=U'(0)=H_p^{-1}(0).
\end{equation}
One can check that
\begin{equation}\label{J18aaa}
H_{pp}-H_p^3\omega(p)=0
\end{equation}
or equivalently
\begin{equation}\label{J18aa}
\Big(\frac{1}{2H_p^2}\Big)_p+\omega(p)=0.
\end{equation}
Moreover it satisfies the boundary conditions
\begin{equation}\label{M4ca}
\frac{1}{2H_p^2(1)}+H(1)=R,\;\;H(0)=0.
\end{equation}
The problem (\ref{J4a}) has a variational formulation (see \cite{CSS}) and the potential is given by
\begin{equation}\label{M2a}
f(h;\lambda)=\int_{-\Lambda_0/2}^{\Lambda_0/2}\int_0^1\Big (\frac{1+\lambda^2h_q^2}{2h_p^2}-h+R-(\Omega(p)-\Omega(1))\Big)h_pdqdp.
\end{equation}

 Then according to \cite{KL1} there exists a branch of solutions to (\ref{J4a})
\begin{equation}\label{J4ac}
h=h(q,p;t):\Bbb R)\rightarrow C^{2,\alpha}_{pe}(\overline{Q}),\;\;\lambda=\lambda(t):\Bbb R\rightarrow (0,\infty),
\end{equation}
which  has a real analytic reparametrization locally around each $t$.

\section{Bifurcation analysis}

Here we present two equation for finding bifurcation points in $q,p$ variables. One of them is defined by a boundary value problem in a two-dimensional domain and another one including a Dirichlet-Neumann operator is defined one a part of one-dimensional boundary. It was proved in Sect. \ref{SOkt10b} that the Frechet derivatives of operators in these two formulations have the same number of negative eigenvalues and the same dimension of the kernels. So both of them can be used in analysis of the Morse index and the corresponding crossing number. The first bifurcation $2$D problem is useful in analysis of the number of negative eigenvalues and the second one is more convenient for application of general bifurcation results.

Both formulations can be easily extended for study of subharmonic bifurcations, see Sect. \ref{SA22a} In Sect. \ref{SOkt10a} we give an explicit connection between the Frechet derivatives in $(x,y)$ and $(q,p)$ variables. This material is borrowed from \cite{Koz1}, where all missing proofs can be found.


\subsection{First formulation of the bifurcation equation}\label{SA23a}

In order to find bifurcation points and bifuracating solutions we put $h+w$ instead of $h$ in (\ref{J4a}) and introduce the operators
\begin{eqnarray*}
&&{\mathcal F}(w;t)=\Big(\frac{1+\lambda^2(h_q+w_q)^2}{2(h_p+w_p)^2}\Big)_p
-\Big(\frac{1+\lambda^2h_q^2}{2h_p^2}\Big)_q\\
&&-\lambda^2\Big(\frac{h_q+w_q}{h_p+w_p}\Big)_q+\lambda^2\Big(\frac{h_q}{h_p}\Big)_q
\end{eqnarray*}
and
$$
{\mathcal G}(w;t)=\frac{1+\lambda^2(h_q+w_q)^2}{2(h_p+w_p)^2}-\frac{1+\lambda^2h_q^2}{2h_p^2}+w
$$
acting on $\Lambda_0$-periodic, even functions $w$ defined in $Q$. After some cancelations we get
\begin{equation}\label{A23a}
{\mathcal F}=\Big(\frac{\lambda^2h_p^2(2h_q+w_q)w_q-(2h_p+w_p)(1+\lambda^2h_q^2)w_p}{2h_p^2(h_p+w_p)^2}\Big)_p
-\lambda^2\Big(\frac{h_pw_q-h_qw_p}{h_p(h_p+w_p)}\Big)_q
\end{equation}
and
\begin{equation}\label{A23aa}
{\mathcal G}=\frac{\lambda^2h_p^2(2h_q+w_q)w_q-(2h_p+w_p)(1+\lambda^2h_q^2)w_p}{2h_p^2(h_p+w_p)^2}+w.
\end{equation}
Both these functions are well defined  for small $w_p$.
Then the problem for finding solutions close to $h$ is the following
\begin{eqnarray}\label{F19a}
&&{\mathcal F}(w;t)=0\;\;\mbox{in $Q$}\nonumber\\
&&{\mathcal G}(w;t)=0\;\;\mbox{for $p=1$}\nonumber\\
&&w=0\;\;\mbox{for $p=0$}.
\end{eqnarray}

Furthermore, the Frechet derivative (the linear approximation of the functions ${\mathcal F}$ and ${\mathcal G}$) is the following
\begin{equation}\label{J4aa}
Aw=A(t)w=\Big(\frac{\lambda^2h_qw_q}{h_p^2}-\frac{(1+\lambda^2h_q^2)w_p}{h_p^3}\Big)_p-\lambda^2\Big(\frac{w_q}{h_p}-\frac{h_qw_p}{h_p^2}\Big)_q
\end{equation}
and
\begin{equation}\label{J4aba}
{\mathcal N}w={\mathcal N}(t)w=(N w-w)|_{p=1},
\end{equation}
where
\begin{equation}\label{J4ab}
N w=N(t)w=\Big(-\frac{\lambda^2h_qw_q}{h_p^2}+\frac{(1+\lambda^2h_q^2)w_p}{h_p^3}\Big)\Big|_{p=1}.
\end{equation}
The eigenvalue problem for the Frechet derivative, which is important for the analysis of bifurcations of the problem
(\ref{F19a}), is the following
\begin{eqnarray}\label{M1a}
&&A(t)w=\mu w\;\;\mbox{in $Q$},\nonumber\\
&&{\mathcal N}(t)w=0\;\;\mbox{for $p=1$},\nonumber\\
&&w=0\;\;\mbox{for $p=0$}.
\end{eqnarray}

\subsection{Second formulation of the bifurcation equation}\label{SA23b}

There is another formulation for the bifurcating solutions, which is used the Dirichlet-Neumann operator. Let us consider the problem
\begin{eqnarray}\label{F19aa}
&&{\mathcal F}(w;t)=0\;\;\mbox{in $Q$},\nonumber\\
&&w=g\;\;\mbox{for $p=1$},\nonumber\\
&&w=0\;\;\mbox{for $p=0$}.
\end{eqnarray}

We define the operator ${\mathcal S}={\mathcal S}(g;h,t)$ by
\begin{equation}\label{F19ab}
{\mathcal S}(g;t)={\mathcal G}(w;t)|_{p=1},
\end{equation}
where $w$ is the solution of the problem (\ref{F19aa}). Then the equation for bifurcating solutions is
\begin{equation}\label{F19b}
{\mathcal S}(g;t)=0.
\end{equation}
 Here we note that spectral problem for the the Frechet derivative of the left-hand side in (\ref{F19b}) is given by
\begin{eqnarray}\label{F19ba}
&&A(t)w=0\;\;\mbox{in $Q$},\nonumber\\
&&N(t)w-w=\theta w \;\;\mbox{for $p=1$},\nonumber\\
&&w=0\;\;\mbox{for $p=0$}.
\end{eqnarray}
More exactly if we introduce the problem
\begin{eqnarray}\label{F19bb}
&&A(t)w=0\;\;\mbox{in $Q$},\nonumber\\
&&w=g \;\;\mbox{for $p=1$},\nonumber\\
&&w=0\;\;\mbox{for $p=0$},
\end{eqnarray}
then the operator
\begin{equation}\label{M3a}
Sg=(Nw-w)|_{p=1}
\end{equation}
is the Frechet derivative of the operator  (\ref{F19b}). The corresponding spectral problem is
\begin{equation}\label{M3b}
Sg=\theta g.
\end{equation}

To prove solvability of the problem (\ref{F19aa}) we consider first the Dirichlet problem
\begin{eqnarray}\label{F20b}
&&Aw=f,\nonumber\\
&&w=g,\;\;\mbox{for $p=1$}\nonumber\\
&&w=0\;\;\mbox{for $p=0$}.
\end{eqnarray}
Let
$$
a_1=\min_{\overline{Q}} h_p,\;\;a_2=||h||_{C^{2,\alpha}(Q)}.
$$

\begin{proposition}\label{PrM1} (Proposition 2.4, \cite{Koz1}) There exists a positive constant  $C$ depending on $a_1$ and $a_2$ such that if $g\in C^{2,\alpha}_{pe}(\Bbb R)$ and $f\in C^{0,\alpha}_{pe}(Q)$, $\alpha\in (0,\gamma]$,
 then the problem {\rm (\ref{F20b})} has a unique solution
$w\in C^{2,\alpha}_{pe}(Q)$, which satisfies the estimate
$$
||w||_{C^{2,\alpha}_{pe}(Q)}\leq C(||f||_{C^{0,\alpha}_{pe}(Q)}+||g||_{C^{2.,\alpha}_{pe}(\Bbb R)}).
$$
\end{proposition}

Next consider the nonlinear problem
\begin{eqnarray}\label{M7a}
&&{\mathcal F}(w;t)=f\;\;\mbox{in $Q$},\nonumber\\
&&w=g\;\;\mbox{for $p=1$},\nonumber\\
&&w=0\;\;\mbox{for $p=0$}.
\end{eqnarray}
It can be considered as a small perturbation of the problem (\ref{F20b}).
\begin{proposition}\label{PrA8} (Proposition 2.5, \cite{Koz1}) There exist positive number $\delta_*$ and a Constant $C$ depending on $a_1$ and $a_2$ such that if
$$
||f||_{C^{0,\gamma}_{pe}(Q)}+||g||_{C^{2,\gamma}_{pe}(\Bbb R)}\leq \delta\leq\delta_*,
$$
then there exists a unique solution $w\in C^{2,\gamma}_{pe}(Q)$ such that
$$
||w||_{C^{2,\gamma}_{pe}(Q)}\leq C\delta .
$$
\end{proposition}

\subsection{Spectral problems (\ref{M1a}) and (\ref{F19ba}) in $(x,y)$ variables}\label{SOkt10a}

Let
\begin{equation}\label{J17aa}
F(q,p;\lambda)=\Gamma(q,h)h_p.
\end{equation}
Then
\begin{equation}\label{J17b}
AF=-\omega'(p)\Gamma-\lambda^2\Gamma_{xx}-\Gamma_{yy}.
\end{equation}
and
\begin{equation}\label{J18a}
-NF+F=\sqrt{\psi_x^2+\psi_y^2}\rho \Gamma-\lambda^2\psi_x\Gamma_x-\psi_y\Gamma_y,
\end{equation}
where
\begin{equation*}
\rho=
\rho(x;t)=\frac{(1+\lambda^2\psi_x\psi_{xy}+\psi_y\psi_{yy})}{\psi_y(\psi_x^2+\psi_y^2)^{1/2}}\Big|_{y=\eta(x;t)}.
\end{equation*}

Then the following assertion is proved in \cite{Koz1}.

\begin{lemma}\label{CorM1} Let $a=a(x,y)$ be positive, continuous, $\Lambda_0$-periodic function. If $\Gamma=\Gamma(x,y)$ satisfies
 the problem
\begin{eqnarray}\label{J17a}
&&(\lambda^2\partial_x^2+\partial_y^2)\Gamma+\omega'(\psi)\Gamma +\mu a\frac{1}{\psi_y}\Gamma=0\;\;\mbox{in $D_\eta$},\nonumber\\
&&(\lambda^2\nu_x\Gamma_x+\nu_y\Gamma_y-\rho \Gamma)_{p=1}=0\;\;\mbox{on $B_\eta$},\nonumber\\
&&\Gamma=0\;\;\mbox{for $x=0$},
\end{eqnarray}
then
\begin{eqnarray}\label{J18ba}
&&AF=\mu aF,\nonumber\\
&&NF-F=0\;\;\mbox{for $p=1$},\nonumber\\
&&F=0\;\;\mbox{for $p=0$}.
\end{eqnarray}
\end{lemma}


\subsection{Comparison of two spectral problems}\label{SOkt10b}

Here we will compare the nagative spectrum of the following spectral problems
\begin{eqnarray}\label{F20a}
&&A(t)u=0\;\;\mbox{in $Q$}\nonumber\\
&&N(t) u-u=\theta au\;\;\mbox{for $p=1$}\nonumber\\
&&u=0\;\;\mbox{for $p=0$}.
\end{eqnarray}
and
\begin{eqnarray}\label{F20aa}
&&A(t)u=\mu bu\;\;\mbox{in $Q$}\nonumber\\
&&N(t) u-u=0\;\;\mbox{for $p=1$}\nonumber\\
&&u=0\;\;\mbox{for $p=0$},
\end{eqnarray}
where $a$ and $b$ are continuous, positive, $\Lambda$-periodic functions.

\begin{proposition}\label{Pm23} The spectral problems {\rm (\ref{F20a})} and {\rm (\ref{F20aa})} have the same number of negative eigenvalues (accounting their multiplicities). Moreover, if we consider these spectral problems problems on $M\Lambda$-periodic, even functions they have the same number of negative eigenvalues also.
\end{proposition}
\begin{proof} For $a=1$ and $b=1$ this assertion is proved in \cite{Koz1}. In general case the proof is literally the same.

\end{proof}

\subsection{Equation for subharmonic bifurcations}\label{SA22a}

For positive integer  $M=1,2,\ldots$, and $\alpha\in (0,1)$ let us introduce  the subspaces $C_{M,e}^{k,\alpha}(\overline{Q})$ and  $C_{M,e}^{k,\alpha}(\Bbb R)$ of $C^{k,\alpha}(\overline{Q})$ and  $C^{k,\alpha}(\Bbb R)$ respectively consisting of even functions of period $M\Lambda_0$. A similar space for the domain ${\mathcal D}_\xi$ we denote by $C_{M,e}^{2,\alpha}(\overline{{\mathcal D}_\xi})$.

Equation (\ref{F19b}) can be considered also on functions of period $M\Lambda_0$. Having this in mind we write as before $h+w$ instead of $h$ but now we assume that $w\in C^{2,\alpha}_{M,e}(\overline{Q})$. Now the operator ${\mathcal F}(w;h,t)$, ${\mathcal G}(w;h,t)$, $A(t)$ and $N(t)$ from Sect.\ref{SA23a} are considered on functions from $C^{2,\gamma}_{Me}(\overline{Q})$. To indicate this difference we will use the notations
${\mathcal F}_M(w;h,t)$, ${\mathcal G}_M(w;h,t)$, $A_M(t)$ and $N_M(t)$ for corresponding operators. Moreover we can define analogs of the operators ${\mathcal S}$ and $S$ and denote them by ${\mathcal S}_M$ and $S_M$ respectively. New operators are acting on functions from $C^{1,\alpha}_{M,e}(\Bbb R)$. The equation for subharmonic bifurcations is
\begin{equation}\label{A23a}
{\mathcal S}_M(g;t)=0,
\end{equation}
where
\begin{equation}\label{M5b}
{\mathcal S}_M(g;t)\,:\, {\mathcal U}\rightarrow  C_{M,e}^{1,\alpha}(\Bbb R).
\end{equation}
Here ${\mathcal U}$ is a neighborhood oh $0$ in the space $C_{M,e}^{2,\alpha}(\Bbb R)$.
The operator ${\mathcal S}_M$ is also potential and for the definition of the potential the integration in (\ref{M2a}) must be taken over the interval $(-M\Lambda_0/2,M\Lambda_0/2)$.
The corresponding eigenvalue problem for the Frechet derivative $S_M(t)$ is
\begin{equation}\label{A23aa}
S_M(t)g=-\theta g,
\end{equation}
where $S_M$ is defined by the same formulas as $S$ but now the functions $g$ and $w$ are $M\Lambda_0$-periodic. Clearly,
\begin{equation}\label{M5ba}
S_M(t)\,:\,C_{M,e}^{2,\alpha}(\Bbb R)\rightarrow  C_{M,e}^{1,\alpha}(\Bbb R).
\end{equation}





\section{Spectral properties of the Frechet derivative, $t$ dependence.}

Since the Frechet derivatives now depends on $t\in [0,\infty)$, we introduce the notations $\mu_j(t)$ and $\widehat{\mu}_j(t,\tau)$ in order to indicate the dependence od the eigenvalues on the variable $t$. Certainly we assume that
$$
\mu_0(t)<\mu_1(t)\leq \mu_2(t)\leq\cdots\;\;\mbox{and}\;\;\widehat{\mu}_0(t,\tau)<\widehat{\mu}_1(t,\tau)\leq \widehat{\mu}_2(t,\tau)\leq\cdots
$$
Clearly $\widehat{\mu}_j(t,0)=\mu_j(t)$.
Using Proposition \ref{Pr1}, we conclude $\mu_0(t)<0$ for all $t$. In what follows we will assume that
$$
\mu_1(t_0)=0,\;\;\mu_1(t)>0\;\;\mbox{for $t\in (0,t_0)$ and}\;\;\mu_1(t)<0\;\;\mbox{for $t\in (t_0,t_0+\varepsilon)$},
$$
where $\varepsilon$ is a small positive number.

\begin{lemma}\label{LF7a} The functions
\begin{equation}\label{F7a}
\widehat{\mu}_1(t,\tau)+\widehat{\mu}_2(t,\tau)\;\;\mbox{ and}\;\;\widehat{\mu}_1(t,\tau)\widehat{\mu}_2(t,\tau)
\end{equation}
 are analytic in a neighborhood of the point $t=t_0$, $\tau=0$.
\end{lemma}
\begin{proof} We introduce the domain
$$
{\mathcal D}=\{u\in H^2_{0,p}(Q)\,:\,N(t) u-u=0\;\mbox{for $p=1$}\}
$$
and consider the operator $A(t,\tau)$ defined on this domain. If $t=t_0$ and $\tau=0$ then the spectrum of this operator in a neighborhood of $(t_0,0)$ consists of the point $0$ and its multiplicity is two.  For $(t,\tau)$ close to $(t_0,0)$ we introduce the spectral projector $P(t,\tau)$ of the operator $A(t,\tau)$ corresponding to the eigenvalues located close to the point  $0$. Its kernel has dimensions $2$ and it depends analytically on $(t,\tau)$ in a neighborhood of $(t_0,0)$. The operator $P(t,\tau))A(t,\tau)$ has finite rank and its trace is invariant with respect to the choice of orthogonal normalized basis. Therefore the functions (\ref{F7a})
are analytical with respect to $t$ and $\tau$ in a neighborhood of $(t_0,0)$.
\end{proof}

\begin{corollary} The functions $\widehat{\mu}_1(t,\tau)$ and $\widehat{\mu}_2(t,\tau)$ are Lipschitz continuous in a neighborhood of $(t_0,0)$.

\end{corollary}
\begin{proof} Consider the equation
$$
\mu^2-B(t,\tau)\mu+A(t,\tau)=0,\;\;A=\widehat{\mu}_1\widehat{\mu}_2,\;\;B=\widehat{\mu}_1\widehat{\mu}_2.
$$
By Lemma \ref{LF7a} the coefficients $A$ an $B$ are analysic with respect to $(t,\tau)$ in a neighborhood of $(t_0,0)$. Since $\mu=\widehat{\mu}_1$ and $\mu=\widehat{\mu}_2$ are roots of this equation the result follows from \cite{KuPa}, for example.

\end{proof}

\subsection{Generalized spectral problem, $\mu=0$.}

Here we study bounded solutions of the problem (\ref{Sept11a}) with $\mu=0$, i.e.
\begin{eqnarray}\label{F4a}
&&\Delta w+\omega'(\psi)w=0 \;\;\mbox{in ${\mathcal D}_\eta$},\nonumber\\
&&\partial_\nu w-\rho w=0\;\;\mbox{on ${\mathcal S}_\eta$},\nonumber\\
&&w=0\;\;\mbox{for $y=0$}.
\end{eqnarray}
According to Lemma \ref{CorM1}  this problem in $q,\,p$ variables has the form
\begin{eqnarray}\label{M5c}
&&A(\tau,t)w=0\;\;\mbox{in $Q$}\nonumber\\
&&N(\tau,t)w+w=0 \;\;\mbox{for $p=1$}\nonumber\\
&&w=0\;\;\mbox{for $p=0$}.
\end{eqnarray}
where
$$
A(\tau,t)w=e^{-i\tau q}A(t)(e^{i\tau q}w), \;\;N(t,\tau)w=e^{-i\tau q}N(t)(e^{i\tau q}w)
$$
We are interested in such real $\tau$ for which the problem (\ref{M5c}) has non-trivial solutions when $|\tau|$ and $|t-t_0|$ are small.

\begin{proposition} Let $t_0>0$ be the first value of $t$ when $\mu_1(t_0)=0$ and $\mu_1(t)<0$ for $t\in (t_0,t_0+\varepsilon)$, where $\varepsilon$ is a small positive number. Then there exist a small positive $\delta$ such that all $\tau$ eigenvalues of (\ref{M5c}) subject to $|t-t_0|<\delta$ and $|\tau|<\delta$ are described by
\begin{equation}\label{D15aa}
\widehat{\tau}_k(s)=\sum_{j=0}^\infty a_{kj}s^{j/n_k},\;\;k=1,\ldots L,
\end{equation}
where $s=t-t_0$. This formula gives all possible small $\tau$-eigenvalues. Here $L$ is the algebraic multiplicity of the $\tau$ eigenvalue $\tau=0$ when $t=t_0$.
\end{proposition}
\begin{proof}
We can reduce (\ref{M5c}) to a standard eigenvalue problem by introducing
$$
v=\frac{(w_q+i\tau w)}{h_p}-\frac{h_qw_p}{h_p^2}.
$$
Then
$$
Aw=\Big(\frac{\lambda^2h_q}{h_p}v-\frac{w_p}{h_p^2}\Big)_p-\lambda^2(i\tau+\partial_q)v
$$
and
$$
Nw=\frac{\lambda^2h_q}{h_p}v-\frac{w_p}{h_p^2}.
$$
The problem (\ref{M5c}) can be written as
\begin{eqnarray}\label{M9a}
&&i\tau w=\frac{h_qw_p}{h_p}-h_pv-w_q\nonumber\\
&&i\tau v=\lambda^{-2}\Big(\frac{\lambda^2h_q}{h_p}v-\frac{w_p}{h_p^2}\Big)_p-\partial_qv
\end{eqnarray}
with boundary conditions
\begin{eqnarray}\label{M9aa}
&&\frac{\lambda^2h_q}{h_p}v-\frac{w_p}{h_p^2}=0\;\;\mbox{for $p=1$}\nonumber\\
&&w=0\;\;\mbox{for $p=0$}.
\end{eqnarray}

We introduce the operator
$$
{\mathcal A}\left (
\begin{array}{ll}
w \\
u \end{array}
\right )=\left (
\begin{array}{ll}
\frac{h_qw_p}{h_p}-w_q & -h_pv \\
-\lambda^{-2}\Big(\frac{w_p}{h_p^2}\Big)_p-\partial_qv &\Big(\frac{h_q}{h_p}v\Big)_p -\partial_qv
\end{array}
\right )
$$
Then
$$
{\mathcal A}={\mathcal A}(t):C^{2,\gamma}\times C^{1,\gamma}\rightarrow C^{1,\gamma}\times C^{0,\gamma}
$$
We put
$$
{\mathcal D}=\{(w,v)\in C^{2,\gamma}\times C^{1,\gamma}\,:\,\mbox{(\ref{M9aa}) is satisfied}\}.
$$
Then the operator
$$
{\mathcal A}:{\mathcal D}\rightarrow C^{1,\gamma}\times C^{0,\gamma}
$$
is Fredholm with zero index and its spectrum consists of isolated $\tau$-eigenvalues of finite algebraic multiplicity. Denote by $X_0$ the kernel of this operator for $t=t_0$ and by $P(t)$ the spectral projector corresponding to the eigenvalues of ${\mathcal A}(t)$ located near $0$. Let also $N_0=\dim X_0$. Then $P(t)$ is the spectral projector of rank $N_0$ depending analytically on $t$ in a small neighborhood of $t_0$. Since the characteristic equation for the operator
 $P(t){\mathcal A}(t)$ is well defined and is invariant with respect of the choice of the basis, the coefficient of the characteristic equation  analytically depends on $t$ in a neighborhood of $t_0$. Therefore the its coefficients
analytically depend on $t$ also.
Therefore we obtain a polynomial equation for $\tau$-eigenvalues located in a small neighborhood of $0$ with coefficients analytically depending on $t$.
According to \cite{VaTr} such equation has roots (\ref{D15aa})
which give all possible small $\tau$-eigenvalues.
\end{proof}

Assume that $t_0>0$ is the first Stokes bifurcation point, which means
\begin{equation}\label{M18aa}
\mu_1(t_0)=0, \;\;\mu_1(t)>0\;\;\mbox{for $t\in (0,t_0)$ and}\;\;\mu_1(t)<0\;\;\mbox{for $t\in (t_0,t_0+\epsilon)$},
\end{equation}
where $\epsilon$ is a small positive number.

Since $\widehat{\mu}_0\leq \nu_0^{*0}$ according to (\ref{Ja19a}), we have
\begin{equation}\label{M16ba}
\widehat{\mu}_0(t,\tau)<0\;\;\mbox{for all $t\geq 0$, $\tau\in\Bbb R$}
\end{equation}
by (\ref{F1a}). Using (\ref{Ja25a}) and (\ref{M16a}), we get
\begin{equation}\label{M16b}
\widehat{\mu}_3(t,\tau)>0\;\;\mbox{for all $t\geq 0$, $\tau\in\Bbb R$},
\end{equation}
provided $\mu_2(t)>0$ for $t\in (t_0,t_0+\epsilon)$.
Furthermore, by (\ref{J6aa}) and Proposition \ref{P22}
\begin{equation}\label{M16bb}
\widehat{\mu}_1(t,\tau)<0\;\;\mbox{for $t\in (0,t_0+\epsilon)$ and  $\tau\in (0,\tau_*)$},
\end{equation}
where $\epsilon$ is chosen to satisfy $\nu_1^{*0}(t)>0$  for $t\in (t_0,t_0+\epsilon)$. Since $\nu_1^{*0}(t)>\mu_1(t)$,  the inequality (\ref{M16bb}) is always valid for $t\leq t_0$ and its validity for $t\in (t_0,t_0+\epsilon)$ requires a restriction on $\epsilon$. Therefore all small $\tau$ eigenvalues of the problem (\ref{M5c}) in the interval $(0,\tau_*)$ for $t\in (t_0,t_0+\epsilon)$ are given by the equation $\widehat{\mu}_2(t,\tau)=0$, where $t\in (t_0,t_0+\epsilon)$. Here we have used inequality (\ref{J6aa}), which implies
\begin{equation}\label{M19a}
\widehat{\mu}_2(t,\tau)>0\;\;\mbox{for $t\leq t_0$ and $\tau\in (0,\tau_*)$}.
\end{equation}

Since $\tau$ eigenvalues are roots of the equation $\widehat{\mu}_2(t,\tau)=0$, there are roots among (\ref{D15aa}) with real coefficients. Moreover if $\tau(t)$ is a real root then $-\tau(t)$ is also a real root. We numerate all positive real roots by $\widehat{\tau}_k(t)$, $k=1,\ldots,m$,
\begin{equation}\label{F6a}
0<\widehat{\tau}_1(t)<\cdots<\widehat{\tau}_m(t),\;\;\mbox{for $t\in (t_0,t_0+\epsilon)$}.
\end{equation}
 We note that since all roots in (\ref{F6a}) satisfies $\widehat{\mu}_2(t,\tau_j(t))=0$, $j=1,\ldots,m$, they are different and we have strong inequalities in (\ref{F6a}). Since the leading term of functions $\widehat{\tau}_j$ is positive the sign of the derivativ is positive also in a small neigborhood of $t_0$. So we will assume that $\varepsilon$ is chosen to satisfy
\begin{equation}\label{M18a}
\partial_t\widehat{\tau}_j(t)>0\;\;\mbox{for $t\in (t_0,t_0+\varepsilon)$}.
\end{equation}

\begin{lemma}\label{LF7aa} Let (\ref{M18aa}) hold. For each $j=1,\ldots,m$ there exists an integer $n>0$ and $\epsilon_1>0$
 such that
  \begin{equation}\label{F6bax}
\partial^{k}_t\widehat{\mu}_2(t,\tau)=0 \;\;\mbox{for $\tau=\widehat{\tau}_j(t)$ and for $t\in (t_0,t_0+\epsilon_1)$ }
\end{equation}
for $k<n$ and
 \begin{equation}\label{F6ba}
\partial^{n}_t\widehat{\mu}_2(t,\tau)\neq 0 \;\;\mbox{for $\tau=\widehat{\tau}_j(t)$ and for $t\in (t_0,t_0+\epsilon_1)$ }.
\end{equation}

\end{lemma}
\begin{proof}  We note that by (\ref{M16ba}) the function $\widehat{\mu}_2(t,\tau)$ is analytic with respect to $\tau$  near a root $\tau=\widehat{\tau}_j$. Let $n=1,2,\ldots$. If the equation
\begin{equation}\label{F6aa}
\partial^n_\tau\widehat{\mu}_2(t,\widehat{\tau}_j(t))=0
\end{equation}
has infinitely many roots in the interval $(t_0,t_0+\epsilon)$ then the equation (\ref{F6aa}) is valid for all $t\in (t_0,t_0+\epsilon)$. If (\ref{F6aa}) holds for a certain $t\in (t_0,t_0+\epsilon)$ and all $n$ then due to analyticity of $\widehat{\mu}_(t,\tau)$ with respect to $\tau$ we get that $\widehat{\mu}_2(t,\tau)=0$ for all small $\tau$ in a neighborhood of $\widehat{\tau}_j(t)$, but this is not true by
\cite{ShaSo}.

Let us choose the smallest $n$ such that $\partial_t^n\widehat{\mu}_2(t,\tau)|_{\tau=\widehat{\tau}_j(t)}\neq 0$ for certain $t$. Then we can choose $\epsilon_1$ for which
\begin{equation}\label{M18b}
\partial^n_\tau\widehat{\mu}_2(t,\tau)|_{\tau=\widehat{\tau}_j(t)}\neq 0\;\;\mbox{for $t\in (t_0,t_0+\epsilon_1)$}.
\end{equation}
Differentiating $\widehat{\mu}_2(t,\widehat{\tau}_j(t))=0$ $n$ times and using (\ref{M18a}), we arrive at (\ref{F6ba}).
\end{proof}


\section{Local subharmonic bifurcation theorem}

In this section we assume that $\mu_1(t_0)=0$ for certain $t_0>0$ and (\ref{M18aa}) holds.

Let
$$
\widehat{a}_1=\widehat{a}_1(t_0,\epsilon)=\min_{\overline{Q}\times [t_0-\epsilon,t_0+\epsilon]}h_p(p,q;t),\;\;\widehat{a}_2=\widehat{a}_2(t_0,\epsilon)=\sup_{ [t_0-\epsilon,t_0+\epsilon]}||h(\cdot,\cdot;t)||_{C^{2,\alpha}(Q)}.
$$
Denote
$$
{\mathcal A}_M=\{ g\in C_{p,M\Lambda_0}^{2,\alpha}(\Bbb R)\;:\; ||g||_{C^{2,\alpha}(\Bbb R)}\leq \delta=\delta(\widehat{a}_1,\widehat{a}_2)\},
$$
where $M$ is a positive integer.
Consider operator ${\mathcal S}_M$ defined on the set ${\mathcal A}_M$ for each $t\in (t_0-\epsilon,t_0+\epsilon)$,  i.e
$$
{\mathcal S}_M\,:\,{\mathcal A}_M\times  (t_0-\epsilon,t_0+\epsilon) \to C_{p,M\Lambda_0}^{1,\alpha}(\Bbb R).
$$
Here we consider the equation (\ref{A23a}) for subharmonic bifurcations. This operator satisfies
$$
{\mathcal S}_M\in C({\mathcal A}_M\times  (t_0-\epsilon,t_0+\epsilon), C_{p,M\Lambda_0}^{1,\alpha}(\Bbb R))
$$
and
$$
S_M\in  C({\mathcal A}_M\times  (t_0-\epsilon,t_0+\epsilon), L(C_{p,M\Lambda_0}^{2,\alpha}(\Bbb R),C_{p,M\Lambda_0}^{1,\alpha}(\Bbb R))).
$$

\begin{theorem}\label{TF9a} Let $t_0>0$ satisfy {\rm (\ref{M18aa})}.  Then there exists an integer  $M_0$ such that for every $M>M_0$ there exists $\varepsilon_M>0$ such that the branch {\rm (\ref{J4ac})} has $M$-subharmonic bifurcation at $t=t_0+\varepsilon_M$. Moreover the crossing number at these bifurcation point is $1$ and $\varepsilon_M\to 0$ as $M\to\infty$. 
The closure of the set of nontrivial solutions of
$$
{\mathcal S}_M(g;t)=0
$$
 near $(0,t_0+\varepsilon_M)$ contains a connected component to which $(0,t_0+\varepsilon_M)$ belongs.

\end{theorem}
\begin{proof}  We use the bifurcation equation (\ref{A23a}) in a neighborhood of $t=t_0$. First we observe that $t_0$ is a Stokes bifurcation point. We choose an  integer $M$.
In order to apply Theorem II,4,4 from \cite{Ki1}, it is sufficient to verify the following spectral properties of the Frechet derivative $S_M(t)$:

(i) find $t_M>t_0$ close to $t_0$ such that the kernel of the operator $S_M(t_M)$ is non-trivial;

(ii) There exists a positive $\varepsilon$ such that the kernel of $S_M(t)$ is trivial for $t\in (t_M,t_M-\varepsilon)\cup (t_M,t_M+\varepsilon)$.
Furthermore
$$
n_M(t_2)-n_M(t_1)=1\;\;\mbox{for cerain $t_2\in (t_M,t_M+\epsilon)$ and $t_1\in (t_M,t_M-\varepsilon)$}.
$$

Due to Proposition \ref{Pm23} it is sufficient to verify properties (i) and (ii) for the branch of spectral problems (\ref{F20aa}) (or (\ref{J17a})) defined on $M\Lambda_0$ ($M\Lambda(t)$) periodic functions with  eigenvalues $\mu^{M}$. We choose a sufficiently large integer $M_0$ and a small $\delta_0>0$, and suppose that $M>M_0$   and  $t_0<t<t_0+\delta_0$. The size of $M_0$ and $\delta_0$ will be clarified below.

We may assume that  $\delta_0\leq \epsilon$ which guarantees the validity of (\ref{M16ba})-(\ref{M16bb}). The inequalities (\ref{M16ba}) and (\ref{M16b}) show that the the only $\tau$ eigenvalues which may contribute to the kernel of the operator $S_M(t)$ come from equations $\widehat{\mu}_1(t,\tau)=0$ or $\widehat{\mu}_2(t,\tau)=0$. By (\ref{M16bb}) the equation $\widehat{\mu}_1(t,\tau)=0$ has the only root $\tau=0$ on the interval $[0,\tau_*)$ with the corresponding eigenfunction $\psi_x$. Since the function $\psi_x$ is odd this equations does not contribute to the kernel also.

Let us consider the equation $\widehat{\mu}_2(t,\tau)=0$. We assume that $\delta_0$ is chosen such that the eigenvalue $\mu=0$ of the problem (\ref{Sept7a})  for $\tau=0$ and $t\in (t_0,t_0+\delta_0)$  is simple \footnote{We recall  that this problem is considered on $\Lambda$-periodic solution and it is not a spectral point for $t\in (t_0,t_0+\delta_0)$}.
By Lemma \ref{L16s}
  \begin{equation}\label{M19b}
  \widehat{\mu}_2(t,\tau)<0\;\; \mbox{for all $t\in (t_0,t_0+\delta_0)$ and small $|\tau|$}.
  \end{equation}
    Applying (\ref{J6aa}) to the eigenvalue $\widehat{\mu}_2(t_0,\tau)$ we get $\widehat{\mu}_2(t_0,\tau)>0$ for $\tau\in (0,\tau_*)$. Using  Lemma \ref{LD2a}, we may assume that $\delta_0$ is chosen to guarantee
  \begin{equation}\label{M19ba}
  \widehat{\mu}_2(t,\tau)>0\;\; \mbox{for $t\in (t_0,t_0+\delta_0)$ and $\tau\in [\delta_1,\tau_*/2]$},
  \end{equation}
where $\delta_1$ is a small positive number.

Second, consider the sequence of functions (\ref{F6a})
$$
\widehat{\tau}_1(t)<\widehat{\tau}_2(t)<\cdots \widehat{\tau}_m(t).
$$
By Lemma \ref{LF7aa} some of them satisfy (\ref{F6ba}) with the the inequality $>0$ and some of the with the inequality $<0$.
Let $n_*$ be the largest index among $1,\ldots,m$ for which the inequality $<0$ is valid in   (\ref{F6ba})  and hence for indexes $j=n_*+1,\ldots,m$ the inequality $>0$ holds in (\ref{F6ba}). Such index exists otherwise $\widehat{\mu}_2(t,\widehat{\tau}_j)\geq 0$ for $t\in (t_0,t_0+\epsilon)$, which contradicts to (\ref{M19b}) and (\ref{M19ba}).

We choose $t_M=t_0+\varepsilon_M$ satisfying the relation
\begin{equation}\label{F7c}
\widehat{\tau}_{n_*}(t_M)=\tau_*/M.
\end{equation}
Our requirement on $M_0$ is that the equation (\ref{F7c}) is solvable for all $M>M_0$.

(i) Let us check (i). The kernel of the operator $A(t)$ consists of functions (\ref{Sept14a}) with $\tau_j$ satisfying
\begin{equation}\label{F7ca}
\widehat{\tau}_{k}(t_M)=\tau_j:=j\tau_*/M\;\;\mbox{for certain $k=n_*+1,\ldots,m$}.
\end{equation}
 If we take $t\neq t_M$ in a small neighborhood of $t_M$ then the kernel will be trivial since all functions $\widehat{\tau}_j$ are strongly increasing.

(ii) Since the functions $\widehat{\tau}_{k}$, $k=n_*+1,\ldots,m$, satisfy (\ref{F6ba}) with the inequality $>0$ they  do not contribute to the changing of the Morse index at $t_M$. Since the function  $\widehat{\tau}_{n_*}$ satisfies
(\ref{F6ba}) with $<0$ instead of $\neq 0$ the eigenvalue $\widehat{\mu}_2(t_M,\tau)$ changes sign at the point $\tau=\widehat{\tau}_{n_*}(t_M)$.
Hence $\widehat{\tau}_{n_*}$  contributes to the changing of the Morse index at $t_M$  with $+1$.  So the crossing number at $t_M$ is $1$. This proves properties (i) and (ii) and the theorem.

\end{proof}







\section{Global subharmonic bifurcation theorems}

Let
$$
X=C^{2,\alpha}_{0,M\Lambda,e}(Q),\;\;Y=Y_1\times Y_2=C^{0,\alpha}_{0,M\Lambda,e}(Q)\times C^{1,\alpha}_{M\Lambda}(\Bbb R).
$$
Introduce a subset in $\Bbb R\times X$:
$$
{\mathcal O}_{\delta}=\{(t,w)\in \Bbb R\times X\,:\,(h+w)_p>\delta\; \mbox{in}\; \overline{Q}, \delta<\lambda(t)\}.
$$
Then the functional
$$
\Gamma(w,t)=({\mathcal F}(w,t),{\mathcal G}(w,t))\,:\,{\mathcal O}_{\delta}\rightarrow Y=Y_1\times Y_2,
$$
where ${\mathcal F}$ and ${\mathcal G}$ are defined by (\ref{A23a}) and (\ref{A23aa}), is well defined and continuous.

The Frechet derivative at the point $w$ has the form
\begin{equation}\label{J4aa}
A\xi=A(w;t)\xi=\Big(\frac{\lambda^2(h+w)_q\xi_q}{(h+w)_p^2}-\frac{(1+\lambda^2(h+w)_q^2)\xi_p}{(h+w)_p^3}\Big)_p
-\lambda^2\Big(\frac{\xi_q}{(h+w)_p}-\frac{(h+w)_q\xi_p}{(h+w)_p^2}\Big)_q
\end{equation}
and
\begin{equation}\label{J4aba}
{\mathcal N}\xi={\mathcal N}(w;t)\xi=(N \xi-\xi)|_{p=1},
\end{equation}
where
\begin{equation}\label{J4ab}
N \xi=N(w;t)\xi=\Big(-\frac{\lambda^2(h+w)_q\xi_q}{(h+\xi)_p^2}+\frac{(1+\lambda^2(h+w)_q^2)\xi_p}{(h+w)_p^3}\Big)\Big|_{p=1}.
\end{equation}
The eigenvalue problem for the Frechet derivative, which is important for the analysis of bifurcations of the problem
(\ref{F19a}), is the following
\begin{eqnarray}\label{M1a}
&&A(w;t)\xi=\mu \xi\;\;\mbox{in $Q$},\nonumber\\
&&{\mathcal N}(w;t)\xi=0\;\;\mbox{for $p=1$},\nonumber\\
&&\xi=0\;\;\mbox{for $p=0$}.
\end{eqnarray}

Introduce some sets:
$$
{\mathcal S}_{\delta}=\mbox{closure in $\Bbb R\times X$ of}\;\;\{(t,w)\in {\mathcal O}_{\delta}\,:\,\Gamma(w,t)=0,\;w\;\;\mbox{is not identically zero}\}.
$$
Let also ${\mathcal C}'_{\delta}$ be the connected component of ${\mathcal S}_{\delta}$ containing the point $(t_M,0)$ together with the connected set from Theorem \ref{TF9a}.

The following theorem is an analog of Theorem 4.2 \cite{CSst}.

\begin{theorem}\label{T4.2C}
Let $\delta>0$. Then either

{\rm (i)} ${\mathcal C}'_{\delta}$ is unbounded in $\Bbb R\times X$, or

{\rm (ii)} ${\mathcal C}'_{\delta}$ contains another trivial point $(t_1,0)$ with $t_1\neq t_M$, or

{\rm (iii)} ${\mathcal C}'_{\delta}$ contains a point $(t,w)\in\partial {\mathcal O}_{\delta}$.

\end{theorem}

In \cite{KL1} it is proved an estimate $\Lambda(t)\geq c_0>0$ for all $t\in \Bbb R$, where the constant $c_0$ depends only on $\omega$ and $r$. This estimate implies
$$
\lambda(t)\leq\frac{\Lambda(0)}{c_0}.
$$

\subsection{Proof of Theorem \ref{T4.2C}}

The proof of the theorem is actually the  same  as that of Theorem 4.2 \cite{CSst}. 

We have
\begin{eqnarray*}
&&{\mathcal F}=-\frac{1+\lambda^2(h_q+w_q)^2}{(h_p+w_p)^3}(h_{pp}+w_{pp})-\lambda^2\frac{h_{qq}+w_{qq}}{h_p+w_p}
+2\lambda^2\frac{(h_q+w_q)(h_{qp}+w_{qp})}{(h_p+w_p)^2}\\
&&=\frac{1}{(h_p+w_p)^3}F(w,t),
\end{eqnarray*}
where
$$
F(w,t)=-(1+\lambda^2(h_q+w_q)^2)(h_{pp}+w_{pp})-\lambda^2(h_p+w_p)^2(h_{qq}+w_{qq})
+2\lambda^2(h_q+w_q)(h_p+w_p)(h_{qp}+w_{qp})
$$
and
\begin{eqnarray*}
&&{\mathcal G}=\frac{1}{(h_p+w_p)^2}G(w,t),
\end{eqnarray*}
where
$$
G(w,t)=1+\lambda^2(h_q+w_q)^2+2(w+h-R)(h_p+w_p)^2.
$$
Since the functions (\ref{J4ac}) solve (\ref{J4a}) we have
$$
F(0,t)=0,\;\;G(0,t)=0\;\;\mbox{for all $t$}.
$$
Therefore the system $\Gamma(w,t)=0$ has the same solutions as the system $(F,G)=0$ provided $h_p+w_p>0$ inside $\overline{Q}$. Since the arguments in
the proof of Theorem 4.2 \cite{CSst} were based only on the periodicity of functions and ellipticity of the problem $(F,G)=0$ the proof given  in \cite{CSst} can be used with small changes to prove  Theorem \ref{T4.2C}.

\subsection{Main Theorem on global subharmonic bifurcations}

Let
$$
{\mathcal C}'=\bigcup_\delta {\mathcal C}'_\delta
$$
The sets ${\mathcal C}'_\delta$
 increase as $\delta$ decreases.

Let
\begin{equation}\label{A17a}
\widehat{\psi}(X,Y;t),\;\widehat{\xi}(X;t),\;\Lambda(t)
\end{equation}
be the elements from ${\mathcal C}'$ in $(X,\,Y)$-coordinates. We will denote the set pf such functions also by ${\mathcal C}'$.
Let
\begin{equation}\label{A21a}
\widehat{\Psi}(X,Y;t)=\Psi(X,Y;t)+\widehat{\psi}(X,Y;t),\;\;\Xi(X;t)=\xi(X;t)+\widehat{\xi}(X,Y;t)
\end{equation}

The main result of this section is the following

\begin{theorem}\label{ThA22a} The following alternatives are valid for the set {\rm (\ref{A17a})}:

{\rm (i)} $\sup_{{\mathcal C}'}\sup_{X\in\Bbb R}|\Xi'(X)|=\infty$, or

{\rm (ii)} $\sup_{{\mathcal C}'}\max_{X\in\Bbb R}(R-\Xi(X))=0$, or

{\rm (iii)} $\sup_{{\mathcal C}'}\sup_{X\in\Bbb R}|\widehat{\Psi}_Y(X,0)|=0$, or

{\rm (iv)} there exists $\delta>0$ such that {\rm (ii) from Theorem \ref{T4.2C}} is valid.

\end{theorem}





\subsection{Proof of Theorem \ref{ThA22a}}

The proof will use the following assertion, which is proved in \cite{KL1}, Proposition 3.2.

\begin{proposition}\label{P3.2} Assume that $\omega\in C^{1,\alpha}([0, 1])$. Let  $\delta > 0$ be given as well as
a ball $B$ of radius $\rho > 0$ and let $M = \sup_{(X,\xi(X))\in B} |\xi'(X)|$. Then there exist constants $\widehat{\alpha}\in (0, 1)$
and $C > 0$, depending only on $R$, $\delta$, $\rho$ and $M$ such that any solution $(\Psi,\xi)\in C^{2,\alpha}(\overline{D_\xi})\times C^{2,\gamma}(\Bbb R)$ of {\rm (\ref{K2a})} with $\inf_B \Psi_Y \geq \delta$ satisfies $||\Psi||_{C^{3,\widehat{\alpha}}(D_\xi\cap \frac{1}{2} B)}\leq C$, where $\frac{1}{2} B$ is a ball with the same centre and
radius $\frac{1}{2} \rho$.

\end{proposition}

The next lemma contains the main step of the proof.

\begin{lemma} Let  $\widehat{\Psi}_Y(X,Y;t)\geq 0$ in ${\mathcal D}_\Xi$ and let there exist positive constants $C_1$, $C_2$, $\delta_1$ and $\widehat{\delta}_1$ such that
\begin{equation}\label{A18a}
|\Xi'(X;t)|\leq C_1,\;\;\Lambda(t)\leq C_2\;\;\mbox{and}\;\;R-\Xi(X;t)\geq \delta_1\;\;\widehat{\Psi}_Y(X,0;t)\geq\widehat{\delta}_1
\end{equation}
for functions from ${\mathcal C}'$.
Then there exist $\delta>0$ such that
\begin{equation}\label{A18aa}
\widehat{\Psi}_Y(X,Y;t)\geq\delta\;\;\mbox{for $(X,Y)\in {\mathcal D}_\Xi$}.
\end{equation}
\end{lemma}
\begin{proof}  The proof consists of several steps.

1. (The first estimate of $\widehat{\Psi}_Y$.) Differentiating $\widehat{\Psi}(X,\Xi (X;t);t)=1$ with respect to $X$, we get
$$
\widehat{\Psi}_X+\widehat{\Psi}_Y\Xi'=0.
$$
Hence
$$
|\widehat{\Psi}_X|\leq C_1|\widehat{\Psi}_Y|\;\;\mbox{on ${\mathcal S}_{\Xi}$}.
$$
From the Bernoulli equation, we get
$$
\delta_1\leq R-\Xi(X)\leq \frac{1}{2}(1+C_1^2)\widehat{\Psi}_Y^2.
$$
Therefore
$$
\widehat{\Psi}_Y^2\geq \delta_2^2:= \frac{2\delta_1}{(1+C_1^2)}.
$$
To estimate the function $\widehat{\Psi}_Y$ on the whole domain ${\mathcal D}_\Xi$ we consider the function
$$
U=\widehat{\Psi}_Y-a\sinh(\beta Y),
$$
where $a$ and $\beta$ are positive constants.
Then
$$
-\Delta U-\omega(\widehat{\Psi})U=a(\beta^2+\omega')\sinh(\beta Y)\geq 0\;\;\mbox{if $\beta^2\geq \max_{p\in [0,1]}|\omega'(p)|$}.
$$
Furthermore, $U(X,0)=0$ and
$$
U(X,\Xi(X))\geq \delta_2-a\sinh(\beta d_-).
$$
We choose $a$ to satisfy
$$
\delta_2-a\sinh(\beta d_-)=0.
$$
By strong maximum principle we get $U\geq 0$ inside ${\mathcal D}_\Xi$. Thus
\begin{equation}\label{A19a}
\widehat{\Psi}_Y(X,Y)\geq a\sinh(\beta Y)\;\;\mbox{in ${\mathcal D}_\Xi$}.
\end{equation}

2. (The estimate of $\widehat{\Psi}$). Let $d_-$ be defined by (\ref{KKKb}). Then by Theorem 1.1 from \cite{KN11a},
$$
\xi(X)\geq d_-\;\;\mbox{for all $X\in\Bbb R$ and for all solutions to (\ref{K2a}) from ${\mathcal C}'$}.
$$
We split the domain ${\mathcal D}_\Xi$ as follows
$$
{\mathcal D}_\Xi=\widehat{\mathcal D}_\Xi\bigcup {\mathcal Q}_{d_-},
$$
where
$$
\widehat{\mathcal D}_\Xi= \{(X,Y)\,:\,X\in\Bbb R,\,d_-\leq Y<\Xi(X)\},\;\; {\mathcal Q}_{d_-}=\Bbb R\times (0,d_-).
$$
Using estimate (\ref{A19a}) together with Proposition \ref{P3.2}, we get
\begin{equation}\label{19aa}
||\widehat{\Psi}||_{C^{3,\widehat{\alpha}}(\overline{\widehat{\mathcal D}_\Xi})}\leq C
\end{equation}
Since $\widehat{\Psi}(\cdot,d_-)\in C^{3,\widehat{\alpha}}(\Bbb R)$ we have that $\widehat{\Psi}\in C^{3,\widehat{\alpha}}(\overline{{\mathcal Q}_{d_-}})$.

3. (The estimate of $\Psi_Y$ from below on ${\mathcal D}_{\Xi}$) Let us show that $\widehat{\Psi}_Y\geq \delta_3$ in  ${\mathcal Q}_{d_-}$, where $\delta_3$ is a positive constant.
Indeed, consider all solutions of
$$
\Delta U+\omega'(\widehat{\Psi})U=0\;\;\mbox{in ${\mathcal Q}_{d_-}$}
$$
satisfying (\ref{19aa}), $U(X,d_-)\geq\delta_4$ and $U\geq 0$ inside ${\mathcal Q}_{d_-}$ and
\begin{equation}\label{A21b}
U(X,d_-)\geq \delta_2\;\;\mbox{and}\;\;U(X,0)\geq\widehat{\delta}_1.
\end{equation}
Since the functions from this set cannot be zero inside ${\mathcal Q}_{d_-}$ we get that all functions must be positive in $\overline{\mathcal Q}_{d_-}$. Since the function $\widehat{\Psi}_Y$ also belongs to this set we get the inequality (\ref{A18aa}) on ${\mathcal Q}_{d_-}$. This together with
(\ref{A19a}) gives the proof of (\ref{A18aa}) on ${\mathcal D}_{\Xi}$.
\end{proof}

\bigskip
\noindent
{\bf Proof of Theorem} Asuume that the alternatives (i)-(iii) from Theorem \ref{ThA22a} are not valid. Then the bifurcating solutions satisfy (\ref{A18aa}) and (\ref{19aa}). This implies boundedness of solutions and that they belong to ${\mathcal C}'_\delta$ for certain positive delta. So we can apply Theorem \ref{T4.2C} and it can be  verified that only the alternative (ii) in Theorem \ref{T4.2C} can be applied here. This proves the theorem.

\begin{remark} In \cite{KL1} it was shown that under the condition $R<d_0$, where
$$
d_0=d(s_0)=\int_0^1\frac{d\tau}{\sqrt{s_0^2-2\Omega(\tau)}},
$$
the condition (iii) in Theorem \ref{ThA22a} can be omitted,

\end{remark}

\bigskip
\noindent {\bf Acknowledgements.} The author was supported by the Swedish Research
Council (VR), 2017-03837. The author expresses also his appreciation to a anonymous reviewer for many  useful comments and corrections.

\bigskip
\noindent
{\bf Data availability statement} Data sharing not applicable to this article as no datasets were generated or analysed during the current study.

\section{Conflict of interest}

On behalf of all authors, the corresponding author states that there is no conflict of interest.

\section{References}

{

\end{document}